\documentclass[11pt]{amsart}
\usepackage[left=2.2cm,tmargin=2.5cm,bmargin=2.5cm,right=2.2cm]{geometry}

\usepackage{amsfonts}
\usepackage{amssymb}
\usepackage{hyperref}
\usepackage{stmaryrd}
\usepackage{mathrsfs}
\usepackage[capitalise]{cleveref}
\usepackage{amscd}
\usepackage{physics}

\usepackage{tipa}

\usepackage{tikz}
\usetikzlibrary{matrix,arrows,decorations.pathmorphing,cd}

\theoremstyle{plain}
\newtheorem{thm}{Theorem}[section]
\newtheorem{lem}[thm]{Lemma}
\newtheorem{cor}[thm]{Corollary}
\newtheorem{prop}[thm]{Proposition}
\newtheorem{conj}[thm]{Conjecture}
 
\theoremstyle{definition}
\newtheorem{defi}[thm]{Definition}
      
\theoremstyle{remark}
\newtheorem{remark}[thm]{Remark}

\newcommand{\lemref}[1]{\hyperref[#1]{Lemma \ref*{#1}}}
\newcommand{\thmref}[1]{\hyperref[#1]{Theorem \ref*{#1}}}
\newcommand{\propref}[1]{\hyperref[#1]{Proposition \ref*{#1}}}
\newcommand{\corref}[1]{\hyperref[#1]{Corollary \ref*{#1}}}
\newcommand{\defref}[1]{\hyperref[#1]{Definition \ref*{#1}}}
\newcommand{\remref}[1]{\hyperref[#1]{Remark \ref*{#1}}}
\newcommand{\conjref}[1]{\hyperref[#1]{Conjecture \ref*{#1}}}

\newcommand{\Gal}{\mathrm{Gal}}

\newcommand{\Spec}{\operatorname{Spec}}

\def \rank {{\mathbf r}}

\def \cond {{\mathbf c}}

\def \F {\mathbb{F}}
\def \C {\mathbb{C}} 

\def \Z {\mathbb{Z}}

\makeatletter
\newcommand*{\defeq}{\mathrel{\rlap{%
                     \raisebox{0.27ex}{$\m@th\cdot$}}%
                     \raisebox{-0.27ex}{$\m@th\cdot$}}%
                     =}

\numberwithin{equation}{section}
\makeatother

\makeatletter
\def\@setcopyright{}
\def\serieslogo@{}
\makeatother

\input xy
\xyoption{all} 

\title[Trace functions over Function Fields]{Short sums of trace functions over function fields \\ and their applications}

\author{Will Sawin}\thanks{W.S. served as a Clay Research Fellow while working on this paper.}

\address{Princeton University, Fine Hall, 304 Washington Rd, Princeton NJ 08540, USA}
\email{wsawin@princeton.edu}

\author{Mark Shusterman}\thanks{M.S. is The Dr. A. Edward Friedmann Career Development Chair in Mathematics}

\address{Faculty of Mathematics and Computer Science, Weizmann
Institute of Science, 234 Herzl Street, Rehovot 76100, Israel.}
\email{mark.shusterman@weizmann.ac.il}

\begin{document}

\begin{abstract}

For large enough (but fixed) prime powers $q$, and trace functions to squarefree moduli in $\F_q[u]$ with slopes at most $1$ at infinity, and no Artin--Schreier factors in their geometric global monodromy, we come close to square-root cancellation in short sums.
A special case is a function field version of Hooley's Hypothesis $R^*$ for short Kloosterman sums. 
As a result, we are able to make progress on several problems in analytic number theory over $\F_q[u]$ such as Mordell's problem on the least residue class not represented by a polynomial and the variance of short Kloosterman sums.
 
\end{abstract}

\maketitle

\tableofcontents

\section{Introduction}

\subsection{Prior work over $\mathbb{Z}$}

\subsubsection{Hooley's conjectures on short Kloosterman sums, its variants, and applications}

In the study \cite{Hooley72} of improvements of the Brun--Titchmarsh theorem, Hypothesis $R$ on cancellation in incomplete Kloosterman sums was introduced, stated roughly as follows.
\begin{conj} \label{HypothesisR}

Fix $\epsilon > 0$. For a positive integer $n$, an integer $b$ coprime to $n$, and a pair of positive integers $\nu_1< \nu_2$ we have
\begin{equation} \label{HooleySumEq}
\sum_{\substack{\nu_1 < a \leq \nu_2 \\ \gcd(a,n)=1}} e \left( \frac{b \overline{a}}{n} \right) \ll_\epsilon (\nu_2 - \nu_1)^{\frac{1}{2}} n^\epsilon
\end{equation}
where $\overline{a}$ stands for the multiplicative inverse of $a$ modulo $n$, and $e(x) = e^{2 \pi i x}$.

\end{conj}

\begin{remark}

At times, we shall drop conditions such as $\gcd(a,n)=1$ from our sums, since the presence of the symbol $\overline{a}$ implicitly indicates that non-invertible residues (modulo $n$) are excluded.

\end{remark}

\cite{Hooley72} then goes on to give applications of this conjecture to the distribution of primes in arithmetic progressions. 
\cite{Hooley78} presents a minor variant of \cref{HypothesisR}, termed Hypothesis $R^*$, where a linear phase is added in the exponential, and uses it to give a lower bound on the number of integers $n \leq x$ for which $n^3 + 2$ has a prime factor exceeding $x^{31/30}$, as $x \to \infty$.
It is conceivable that this hypothesis can be used to obtain similar results for other polynomials in $n$, perhpas of arbitrary degree, building on the methods of \cite{Welsh}. 
The applications of Hypothesis $R^*$ to the distribution of primes in arithmetic progressions were further studied in \cite{Iwaniec82}.

Hooley's Hypotheses and their modifications found applications also in other problems.
Assuming this hypothesis, \cite{Bala--Conrey--Heath-Brown85} computes the second moment of the Riemann zeta-function on the critical line twisted by the square of a very long (but otherwise arbitrary) Dirichlet polynomial, building on ideas of \cite{Iwaniec80}.
Conditionally on a form of Hypothesis $R$ with $\overline{a}$ replaced by $\overline{a}^2$ for moduli $n$ that are squares, 
\cite[Theorem 2]{Fouvry16} obtains the lower bound in \cite[Conjecture 1]{Hooley84} on the number of nonsquare $D \leq x$ for which the fundamental solution of the Pell equation $t^2 - Ds^2 = 1$ is bounded from above by $D^{\frac{1}{2} + \alpha}$, for each $\frac{1}{2} < \alpha \leq \frac{2}{3}$.

Variants of these hypotheses, where one is allowed to twist the summand in \cref{HooleySumEq} by a multiplicative character $\chi(a)$, have been considered in \cite{Friedlander--Iwaniec87}.
With this twist, the sum may be viewed as an incomplete Sali\'{e} sum. These variants are then shown to imply the existence of integers $n \leq N$ with $|| \theta n^2 || \ll_\epsilon N^{-\frac{2}{3} + \epsilon}$, for any fixed $\theta \in \mathbb{R}$, any fixed $\epsilon > 0$, and every $N$ large enough, where $||y||$ is the distance in $\mathbb R$ from $y$ to the nearest integer.
 
\subsubsection{Approaches to incomplete exponential sums and their applications}

Several methods have been developed to make progress on the aforementioned hypotheses, resulting in unconditional applications. The P\'{o}lya--Vinogradov completion method, in conjunction with Weil's bound on Kloosterman sums, is capable of giving power savings in \cref{HooleySumEq} but only in case $\nu_2 - \nu_1 > n^{\frac{1}{2} + \epsilon}$. This method is perfected by \cite{Fouvry--Kowalski--Michel--Raju--Rivat--Sound17}.
A folklore application of the resulting bounds on \cref{HooleySumEq} is a level of distribution of $\frac{2}{3}$ for the binary divisor function in arithmetic progressions.

Another classical result resorting to the completion method is the equidistribution of the roots of a quadratic congruence to a random modulus established in \cite[Theorem 3]{Hooley63}. It is implicit in the proof that Hypothesis $R$ can be used to improve the power saving error term, and it is conceivable that it can also be used to obtain a power saving error term for congruences of arbitrary degree.
Yet another example is \cite[Theorem]{Iwaniec78} showing that the number of $n \leq x$ for which $n^2+1$ is a product of at most two primes is $\gg \frac{X}{\log X}$. It is implicit in \cite{Iwaniec78} that Hypothesis $R$ can be used to increase the implied constant.  
A more recent application is \cite[Theorem 1]{Fouvry16}, an unconditional result in the direction of the lower bound in the aforementioned \cite[Conjecture 1]{Hooley84}.

\cite{Heath-Brown01} developed a modular analog of van der Corput's method, obtaining power savings once the modulus $n$ in \cref{HooleySumEq} is `flexibly factorable', and applied it to obtain an unconditional variant of the aforementioned \cite[Theorem 3]{Hooley78} on the prime factors of integers of the form $n^3+2$. 
For extensions of these techniques to quartic polynomials in $n$, improvements on the cubic case, and more information on this problem, we refer to \cite{DM} and references therein.

\cite{Xi18} used Heath-Brown's method, and arguments from \cite{Fouvry16}, to make further progress on the lower bound in the aforementioned \cite[Conjecture 1]{Hooley84}, for each $\frac{1}{2} < \alpha \leq 1$.
\cite{Irving15} used this method to obtain a level of distribution which exceeds $\frac{2}{3}$ for the binary divisor function in arithmetic progressions to squarefree smooth moduli,
and \cite{Xi18b} established an improved level of distribution for the ternary divisor function in such progressions.
 
Karatsuba developed a method which saves a power of $\log(\nu_2 - \nu_1)$ in \cref{HooleySumEq} once $\nu_2 - \nu_1$ is not much smaller than $\nu_1$, see \cite{Korolev16}.
Remarkably, cancellation persists even if the length of summation is as short as $e^{\log^{2/3 + \delta}(n)}$ for some positive $\delta$, which is beyond the scope of \cref{HypothesisR}. 
\cite{Bourgain--Garaev14} refined Karatsuba's method, and used it to obtain an improvement of the Brun--Titchmarsh theorem.
\cite{Bourgain15} applied this method to (yet again) make progress on the lower bound in \cite[Conjecture 1]{Hooley84},
this time for $\alpha$ close to $\frac{1}{2}$.

\cite{Khan16} used Weyl differencing to improve on the cancellation in \cref{HooleySumEq} provided by the completion method in case $n = p^k$ where $p$ is an odd prime, $k \geq 7$ is an integer, and $\nu_2 - \nu_1$ lies in an appropriate range depending on $n$. A further improvement in case $k$ is large enough has been obtained in \cite{Liu--Shparlinski--Zhang18} which relied on Vinogradov's method. As a consequence, both works obtained an improved level of distribution for the binary divisor function in progressions to certain prime power moduli.
We refer to \cite[Chapter 8]{IwKo} for an exposition of the methods of van der Corput, Weyl, and Vinogradov.

\subsubsection{Additional averaging}

As has already been noted in the aforementioned \cite{Friedlander--Iwaniec87}, in some applications of \cref{HypothesisR} and its variants, one only needs to obtain (a certain amount of) cancellation in \cref{HooleySumEq} in the presence of additional summation over either $n$ or $b$. 
With this in mind, \cite{Duke--Friedlander--Iwaniec97} and \cite{Bettin--Chandee18} obtained unconditional substitutes for the conjectures above, allowing for very general forms of averaging over $n$ and $b$. 
These works found applications to many problems in analytic number theory.
Since our interest here will eventually be in function fields, we only mention (some) applications that could be of interest in that setting.
In particular, we do not discuss applications that are known to follow from the Generalized Riemann Hypothesis or the ABC conjecture, since both are theorems over function fields.

\cite{Granville--Shao19} obtained a Bombieri--Vinogradov theorem with a fixed residue class for many bounded multiplicative functions beyond the square-root barrier.
\cite{FR18} obtained a related result for certain unbounded arithmetic functions such as the generalized divisor functions.
These works are based on ideas of \cite{Green18} studying bounded multiplicative functions in arithmetic progressions to almost all prime moduli.
In these applications, numerically stronger (but, conditional) results can be obtained by invoking a form of Hypothesis $R$ instead of the unconditional estimates of \cite{Bettin--Chandee18}, see for instance \cite{Fouvry81}.

\cite{BCR17} used \cite{Bettin--Chandee18} to obtain an unconditional version of the aforementioned result from \cite{Bala--Conrey--Heath-Brown85} on the twisted second moment of the Riemann zeta-function. An analog of \cite{BCR17} for the family of primitive Dirichlet characters with prime modulus has been obtained in \cite{Bui--Pratt--Robles--Zaharescu20}.

Inspired by the Burgess method, \cite{Friedlander--Iwaniec85} devised a technique for reducing an averaged form of \cref{HooleySumEq} over $b$ to the problem of (square-root) cancellation in certain complete exponential sums. This led to a level of distribution exceeding $\frac{1}{2}$ for the ternary divisor function in arithmetic progressions to prime moduli. 
The method was further refined by \cite{Heath-Brown86}, \cite{Fouvry--Kowalski--Michel15}, and \cite{Sharma23} leading to a better level of distribution. 

Another unconditional replacement for Hooley's Hypothesis $R$, with numerous applications, is the work \cite{Deshouillers--Iwaniec82} utilizing the spectral theory of automorphic forms to obtain significant cancellation in the presence of sufficiently nice averaging over the modulus $n$ of a Kloosterman sum.
We refer to the survey \cite{Shparlinski12} for other works on variants of \cref{HypothesisR}, and a variety of further applications.

\subsubsection{More general short exponential sums}

Other incomplete exponential sums have also found applications in analytic number theory.
For example, the sum
\begin{equation} \label{ZhangSumEq}
\sum_{\substack{\nu_1 < a \leq \nu_2}} e \left( \frac{b\overline{a} + c\overline{a+1}}{n} \right)
\end{equation}
which bears a resemblance to \cref{HooleySumEq}, showed up in the study of the level of distribution of primes and divisor-like functions in arithmetic progressions to smooth moduli in \cite{Zhang14} and \cite{Wei--Xue--Zhang16}.
\cite{Polymath14} applied Heath-Brown's method to obtain the state of the art bound on \cref{ZhangSumEq}.
Similar sums appear also in \cite{Hooley78}.

A series of works starting with \cite{Burgess} developed tools to obtain cancellation in sums of the form
\begin{equation} \label{BurgessSum}
\sum_{\nu_1< a \leq \nu_2} \chi(P(a))
\end{equation}
where $\chi$ is a Dirichlet character, and $P \in \mathbb{Z}[X]$ is a polynomial. 
Burgess mostly studied the case of $\chi$ with cube-free conductor, and succeeded the most in case $P$ is linear, quadratic, or a product of linear factors, and $\nu_2 - \nu_1$ exceeds the fourth root of the conductor of $\chi$.
We refer to \cite{Chang09} for a further discussion of the method of Burgess and its refinements, and to \cite{Stepanov--Shparlinski90} for the case the conductor of $\chi$ is a large power of a fixed prime.
For a sample (partly conjectural) application of cancellation in \cref{BurgessSum} see \cite{Friedlander--Iwaniec--Mazur--Rubin13}.

Especially in the case of squarefree moduli, all the sums considered above are special cases of short sums of trace functions, a very general class of periodic functions of algebraic origin, whose significance in number theory is highlighted in the works of Fouvry, Kowalski, Michel, Sawin, and others. 
Another interesting example of a trace function is the (hyper)-Kloosterman function of $a$ given by
\begin{equation} \label{HyperKloostermanOverZexample}
\operatorname{Kl}_k(a;b) = \frac{(-1)^{k-1}}{n^{(k-1)/2}} \sum_{\substack{ x_1, \dots, x_k \in (\mathbb Z/ n \mathbb Z)^\times \\ x_1 \cdots x_k \equiv a \ \mathrm{mod} \ n}} e \left( \frac{b \cdot (x_1 + \dots + x_k)}{n} \right)
\end{equation}
where $k,n \geq 2$ are integers, and $b$ is coprime to $n$.
Short sums of these trace functions for $k=2$ have been studied in the aforementioned works \cite{Khan16} and \cite{Liu--Shparlinski--Zhang18}.
For more examples of trace functions we refer to what follows, and to \cite{Fouvry--Kowalski--Michel15b}.

Various aspects of short sums of trace functions have been studied. For instance, \cite{Perret-Gentil17} and \cite{Harp} consider the distribution of sums of trace functions over all short intervals. Other examples include \cite{Fouvry--Kowalski--Michel14}, and references therein, which study short sums of trace functions over primes and against multiplicative functions. We also mention \cite{Kowalski--Sawin16} where short sums of trace functions play an important role.  

The P\'olya--Vinogradov method is applicable to short sums of an arbitrary (irreducible) trace function.
Heath-Brown's variant of van der Corput's method applies to quite general trace functions, those having `slopes at most $1$ at infinity' (this restriction will reappear in the sequel).
On the other hand, the methods of Burgess and Karatsuba are much more specialized.

\subsection{Our results over function fields}

For a power $q$ of a prime number $p$, we replace the ring of integers $\mathbb{Z}$ by the polynomial ring $\F_q[u]$ in one variable $u$, over the finite field $\F_q$.
For a polynomial $f \in \F_q[u]$ we define its norm by $|f| = |\F_q[u]/(f)| = q^{\deg(f)}$ with the convention that $\deg(0) = -\infty$ so the norm of the zero polynomial is $0$.

\subsubsection{Short sums of trace functions}

We can view an interval in $\mathbb Z$ as the set of integers whose difference with a given integer (the center of the interval) has norm bounded from above by a certain quantity.
The above notion of norm in $\F_q[u]$ then gives us a corresponding definition of an interval in $\F_q[u]$.
In other words, a `short' interval around a polynomial $f \in \F_q[u]$ is the collection of polynomials that differ from $f$ by a `low-degree' polynomial.

\cite[Theorem 2.1]{SS19} gives a bound on the $\F_q[u]$-analog of \cref{BurgessSum} in case $P$ is linear and the conductor of $\chi$ is squarefree, whose quality is that of the bound in \cref{HypothesisR} once $q$ is large enough in terms of $\epsilon$. This short multiplicative character sum bound was then applied to make progress on the level of distribution of the M\"{o}bius and von Mangoldt functions in arithmetic progressions, the Chowla conjecture, and the twin primes problem over $\F_q[u]$ for $q$ satisfying suitable technical hypotheses.

Next we introduce the general notion of trace functions we want to sum in short intervals in $\F_q[u]$, and the assumptions we make on these functions for our main theorem to apply.

Fix throughout an auxiliary prime number $\ell$ different from $p$ and an embedding $\iota \colon \overline{\mathbb Q_\ell} \hookrightarrow \mathbb C$.
We will make tacit use of $\iota$ to view elements of $\overline{\mathbb Q_\ell}$ as lying in $\mathbb C$.
Let $\pi \in \mathbb F_q[u]$ be an irreducible polynomial, and let $\mathcal{F}$ be a (constructible \'etale $\overline{\mathbb Q_\ell}$-)sheaf on $\mathbb A^1_{\mathbb F_q[u]/(\pi)}$.
Its trace function $t = t_{\mathcal F}$ is given by
\begin{equation*} \label{DefTraceFuncFrobEq}
t \colon \mathbb F_q[u]/(\pi) \to \mathbb C, \quad t(x) = \tr(\mathrm{Frob}_{|\pi|}, \mathcal{F}_{\overline{x}}), \qquad x \in \mathbb A^1_{\mathbb F_q[u]/(\pi)}(\F_q[u]/(\pi)) = \F_q[u]/(\pi)
\end{equation*}
where $\overline x$ is a geometric point over $x$.
We often think of $t$ as a $\pi$-periodic function on $\F_q[u]$.

We make the following four assumptions on $\mathcal F$.

\begin{itemize}

\item As a normalization condition on the values of $t$ we assume that $\mathcal F$ is mixed of nonpositive weights. Each value of $t$ is then bounded from above by the (generic) rank of $\mathcal F$ - the dimension of the stalk of $\mathcal F$ at a geometric generic point of $\mathbb A^1_{\F_q[u]/(\pi)}$.
We denote this quantity by $r(t)$.
For the notion of `mixed of nonpositive weights' we refer to \cite[Definition 1.4]{SS20}.

\item We assume that $\mathcal F$ has no finitely supported sections. This minor assumption, which is satisfied in natural examples, means that the values of $t$ `vary naturally' with $x$, with no particular value of $t$ having `additional weight'. We can dispense with this assumption at the cost of making our bounds somewhat more cumbersome.
For the support of sections we refer to \cite[Equation (1.16)]{SS20}.

\item We assume that the geometric global monodromy of $\mathcal F$, a finite-dimensional Galois representation over $\overline{\F_q}(u)$, has no factor of the form $\mathcal L_{\psi}(\alpha x)$ for any $\alpha \in \overline{\F_q}$.
Here $\psi \colon \F_p \to \C^\times$ is a nontrivial additive character, and $\mathcal L_\psi$ is the associated Artin--Schreier sheaf for which we refer to \cite[Section 2.4]{SS20}. For $\alpha = 0$ our assumption means the absence of trivial factors.
This assumption is needed in order to rule out constant sums, and sums of additive characters on subspaces, which may not exhibit any cancellation.

\item We assume that the slopes of the (geometric) local monodromy representation of $\mathcal F$ at $\infty \in \mathbb P^1$, a finite-dimensional Galois representation over $\overline{\F_q}((u^{-1}))$, are bounded from above by $1$. For the notion of slopes we refer to \cite[Section 2.3]{SS20}, and to \cite[Chapter 1]{Katz88} where these are called `breaks'.
This is the most significant/restrictive of our four assumptions. 
It points to, what seems to be, a limitation of our method, and also a limitation for some forms of the van der Corput method.

\end{itemize}

For a squarefree $g \in \F_q[u]$, a trace function $t \colon \F_q[u]/(g) \to \C$ is a function of the form
\begin{equation*} 
t(x) = \prod_{\pi \mid g} t_{\pi} (x \ \mathrm{mod} \ \pi), \quad x \in \F_q[u]/(g)
\end{equation*}
where the product ranges over the monic irreducible factors of $g$, and for each such factor $t_\pi$ is a trace function associated to a sheaf $\mathcal F_\pi$ on $\mathbb A^1_{\F_q[u]/(\pi)}$.

\begin{thm} \label{MainRes}

Suppose that the sheaves $\{\mathcal{F}_\pi\}_{\pi \mid g}$ satisfy the four assumptions above.
Then
\begin{equation*} 
\sum_{\substack{f \in \F_q[u] \\ |f| < X}} t(f) \ll X^{\frac{1}{2}} |g|^{\log_q(2r(t) + c(t))}, \qquad X,|g| \to \infty
\end{equation*}
with the implied constant being $\sqrt q$.

\end{thm}

Here the conductor of $t$, denoted by $c(t)$, is defined as in \cite[Equation (1.17), Definition 1.8]{SS20}.
As shown in the proof of \cite[Lemma 3.10 (6)]{SS20}, this conductor coincides with the Fourier conductor from \cite[Definition 3.8]{SS20}, which is the (generic) rank of the Fourier transform of the sheaf.
\cite[Lemma 3.10 (6)]{SS20} makes the stronger assumption of tame ramification at $\infty$, namely that the slopes are zero, but the proof works also under our weaker assumption that the slopes are at most $1$.
In case the sheaves arise from their Galois representations by middle extension, these conductors equal the logarithm of the Artin conductor of the Galois representation, see \cite[Remark 1.6]{SS20}.

\cref{MainRes} is stated only for (short) intervals around the zero polynomial but by translating the function (by pulling back the sheaves giving rise to it along a translation on the affine line) we can always reduce to this case.

\cref{MainRes} extends both the aforementioned bound on short multiplicative character sums from \cite[Theorem 2.1]{SS19}, and its generalization \cite[Theorem 1.10]{SS20} which assumes tame ramification at $\infty$, and the existence of an irreducible factor $\pi$ of $ g$ for which $t_\pi$ is a (shifted) multiplicative character. 
\cite[Theorem 1.10]{SS20} was applied to make progress on the Bateman--Horn conjecture, and on sums of trace functions against arithmetic functions over $\F_q[u]$.
Moreover, in \cref{GeneralizedHooleyRationalFunction} we will see that \cref{MainRes} applies to the function field versions of \cref{HooleySumEq}, \cref{ZhangSumEq} and \cref{BurgessSum}. In particular we obtain cancellation in the character sums from \cite[4.3]{Sawin20}.

The fourth assumption on the trace function in \cref{MainRes} cannot be removed completely since
\begin{equation*}
\sum_{\substack{f \in \F_q[u] \\ \deg(f) < m}} e \left( \frac{f^2}{g} \right) = q^m, \quad m < \frac{\deg (g) - 1}{2}. 
\end{equation*}
We refer to \cite[Section 2.4.1]{SS20} for the exponential function in $\F_q(u)$.

Even though the assumptions in \cref{MainRes} are, to some extent, similar to those in van der Corput's method, the savings we obtain in \cref{MainRes} are not always better (or always worse) than those provided by the latter method. For intervals that are not too short, our savings are superior once $q$ is sufficiently (but not extremely) large.
Moreover, \cref{MainRes} does not require the modulus $g$ to be smooth.
For smaller values of $q$, or very short intervals, our bounds are worse than trivial, whereas the van der Corput method is capable of establishing cancellation once the modulus is sufficiently smooth.
A similar comparison can be made between our method and the methods of P\'{o}lya--Vinogradov, Burgess, and Karatsuba.

The (hyper)-Kloosterman sheaf giving rise to the analog over $\F_q[u]$ of \cref{HyperKloostermanOverZexample} for squarefree moduli satisfies our four assumptions in view of \cite[Theorem 4.1.1]{Katz88} so \cref{MainRes} applies to it.

Now we specialize \cref{MainRes} to mixed character sums.

\begin{cor} \label{GeneralizedHooleyRationalFunction}

Let $\chi \colon (\F_q[u]/(g))^\times \to \C^\times$ be a primitive multiplicative character.
Let $F \in \F_q[u]/(g)[T]$ be a polynomial, and let $a,b \in \F_q[u]/(g)[T]$ be polynomials such that for every irreducible factor $\pi$ of $g$ the degree (in $T$) of the reduction of $a$ modulo $\pi$ exceeds by at most $1$ the degree (in $T$) of the reduction of $b$ modulo $\pi$ which is less than $p$.
Suppose moreover that for every irreducible factor $\pi$ of $g$ the degree (in $T$) of the reduction of $F$ modulo $\pi$ is positive, or the reduction of $b$ modulo $\pi$ does not divide the reduction of $a$ modulo $\pi$.
Then for $n \leq \deg g$ we have
\begin{equation*}
\left| \sum_{\substack{h \in \F_q[u] \\ \deg h < n}} \chi (F(u,h))  e\left( \frac{a(u,h) \overline{b(u,h)}}{g} \right) \right| \leq  q^{\frac{n+1}{2}} \max\{\deg_T F + 2 \deg_T b + 2,  2 \deg_T b + 2\}^{\deg g}.
\end{equation*}

\end{cor}

The primitivity of $\chi$ means that it does not factor via $(\F_q[u]/(D))^\times$ for any divisor $D$ of $g$.

To get the analog over $\F_q[u]$ of the sum in \cref{HooleySumEq} we take $a$ and $F$ to be the constant polynomial $1$, and $b$ to be linear in $u$.
For \cref{ZhangSumEq} we also put $F = 1$ and take $a/b$ to be a suitable rational function.
For \cref{BurgessSum} we take $a$ to be the zero polynomial and $b$ to be the constant polynomial $1$.

It is possible to prove a version of \cref{GeneralizedHooleyRationalFunction} with less restrictive, albeit more technical, assumptions.
We have opted for a statement which should be good enough for most applications, and is convenient to state and prove.


\subsubsection{Applications}

Here we present two applications of \cref{MainRes} to analytic number theory over $\F_q[u]$.
A plethora of applications can already be obtained by adapting to function fields the reductions to Hypothesis $R$ (or to other exponential sums) in the aforementioned (at times conditional) results over $\mathbb{Z}$.
Achieving a natural level of generality may require additional ideas, not necessarily from \'etale cohomology or the theory of exponential sums, and may be pursued in future works.
Here we chose however to present applications to problems whose reduction to short exponential sums has not (as far as we know) appeared previously in the literature. 
It is plausible that the results in this section can be adapted to results over $\mathbb{Z}$, conditional on cancellation in certain incomplete exponential sums, as in some of the results mentioned earlier in the introduction.
Unconditional progress over $\mathbb{Z}$ can also be made by using one of the aforementioned methods for getting cancellation in short exponential sums.

For a prime number $\lambda$ and a polynomial $P \in (\Z/\lambda\Z)[T]$, \cite{Mordell63} gave upper bounds on the least residue in $\Z/\lambda\Z$ which is not represented by $P$, in case $\deg(P)$ = 3. Mordell's arguments were generalized in \cite{Hudson66} and \cite{Williams66} to cover the case $\deg(P) = 4$,
and a quantitative variant of the problem has been considered in \cite{McCann--Williams67a} and \cite{McCann--Williams67b}.
An extension to polynomials $P$ of arbitrary degree has been obtained in \cite{Bombieri--Davenport66} which established the existence of $m \ll \lambda^{1/2} \log \lambda$ that is not represented by $P$, as soon as $\lambda$ is large enough compared to $\deg(P)$, and $P$ is not a permutation polynomial.

An immediate consequence of \cref{MainRes} and \cite[Proposition 6.7]{Fouvry--Kowalski--Michel14} is a result on the function field analog of Mordell's problem, that is analogous to a bound of the form $m \ll \lambda^\epsilon$ on the least residue not represented by $P$. We state it here in a quantitative form.

\begin{cor} \label{MordellCor}

Let $d$ be a positive integer, let $p>d$ be a prime number, and let $q$ be a power of $p$.
For an irreducible $\pi \in \F_q[u]$, and a polynomial $P \in \F_q[u]/(\pi)[T]$ with $\deg_T P = d$, let $\mathcal{P}$ be the set of values of $P$ mod $\pi$. 
Then for $X$ a power of $q$ we have
\begin{equation*}
\# \{f \in \F_q[u] : |f| < X, \ f \in \mathcal{P} \} \sim \frac{|\mathcal{P}|}{|\pi|} X, \quad |\pi| \to \infty,
\end{equation*}
as soon as $X \gg |\pi|^{3 \log_q(3 \cdot d!)}$.

\end{cor}

Our second application computes the variance of certain trace functions in short intervals.

\begin{thm} \label{VarianceShortTraceSums}

Let $P \in \F_q[u]$ be an irreducible polynomial, and let $t$ be the trace function of a sheaf $\mathcal{F}$ on $\mathbb{A}^1_{\F_q[u]/(P)}$ that has no finitely supported sections, is punctually pure of weight $0$, has an irreducible geometric monodromy representation which is not isomorphic to the geometric monodromy of any nontrivial additive shift of $\mathcal{F}$ and not Artin--Schreier (or constant).
Then for $X \leq |P|$ a power of $q$, we have
\[
\frac{1}{|P|} \sum_{\substack{f \in \F_q[u] \\ |f| < |P|}} \left| \frac{1}{X^{1/2}} \sum_{\substack{g \in \F_q[u] \\ |f-g| < X}} t(g) \right|^2 =1 + O\left(X^{\frac{1}{2}} |P|^{-\frac{1}{2} + 2\log_q(3(r(t) + c(t))) }\right)
\]
where the implied constant depends only on $q$.
\end{thm}

Although the requirement that $\mathcal{F}$ satisfies the assumptions of \cref{MainRes} allows for a variety of trace functions, its appearance is mainly due to the limitations of our methods. The assumptions going beyond that, are made mostly for convenience, and can be removed at the cost of a more elaborate main term and proof. 
We can write the power of $|P|$ in the error term as $|P|^{-\frac{1}{2} + \epsilon}$ with $\epsilon$ becoming smaller as $q$ becomes larger.

The proof proceeds by expanding the square which leads one to consider autocorrelations of $t$ with shifts of itself, on average over the shift.
If one does not attempt to benefit from this additional averaging, and applies Delign'es bound to each autocorrelation separately, the resulting error term is $X|P|^{-1/2}$, which gives an asymptotic for the variance only for $X \ll |P|^{1/2}$. 

Our error term is obtained by utilizing the additional averaging over the shift, encoded in the trace function 
\[
h \mapsto \sum_{f \in \F_q[u]/(P)} t(f) \overline{t(f+h)}.
\]
The proof of \cref{VarianceShortTraceSums} uses Laumon's local Fourier transform to study the trace function above.
One can deduce from \cref{VarianceShortTraceSums} a lower bound for at least one of the short sums of $t$ of length $X$.

\subsubsection{Strategy of the proof of \cref{MainRes}}

As in many other results in function field number theory, including \cite{SS19,SS20}, our bounds for exponential sums arise from vanishing results for the high-degree compactly supported cohomology of a certain \'etale sheaf. \cite{SS19} and \cite{SS20} used two different methods to prove vanishing for high-degree compactly supported \'etale cohomology of the relevant sheaf.

\cite{SS19} used an approach by deformation, where one considers a family of sheaves of which the sheaf is one member, controls the discrepancy between the cohomology of one member of the family and a generic member using vanishing cycles, and then shows the vanishing of the high-degree compactly supported cohomology of a generic member of the family by somehow taking advantage of how the family varies. In \cite{SS19}, this was done by finding a single member of the family whose cohomology could be computed exactly. The family of sheaves is the family of incomplete exponential sums of a given length, and the single member is the sum of a multiplicative character over low degree polynomials (the interval around the zero polynomial), whose estimation boils down to the Riemann Hypothesis for curves over finite fields.

\cite{SS20} used an approach by Artin's affine theorem, which allows one to show vanishing of the high-degree ordinary \'etale cohomology, and one can then control the discrepancy between the ordinary \'etale cohomology and compactly supported \'etale cohomology by studying the derived pushforward of a sheaf to its compactification.

Both of these are fundamental methods in \'{e}tale cohomology that have been applied succesfully to multiple problems. In several respects, they are similar. Both require local study of the sheaf on a compactification (as it is necessary to compactify a family before vanishing cycles may be used to compute its cohomology). In both it is sufficient to prove local properties of the sheaf away from a Zariski closed set, and the number of cohomology groups one can hope to prove vanish depends on the codimension of the closed set. Some crucial differences are that the Artin's affine theorem method usually requires the sheaf to have sufficiently nontrivial monodromy at infinity, while the deformation method does not. On the other hand, the deformation method requires a separate argument to control the generic member of the family, which the Artin's affine theorem method does not.

In this paper we return to the deformation method of \cite{SS19}, but with a new strategy for the details that renders it much more general than \cite{SS20} instead of much less general, like \cite{SS19}. As in \cite{SS19}, the family of sheaves we construct corresponds to the family of short exponential sums given by summing the same trace function over different short intervals of the same length.

To control the generic member of the family, instead of an exact calculation, which does not seem possible in general, we use a Fourier transform method. 
It is straightforward to check that the Fourier transform of the function
\begin{equation} \label{FunctionSelectIntervalShortSum}
h \mapsto \sum_{ \substack{f \in \mathbb F_q[u] \\ |f|<X}} t(f+h), \qquad h \in \mathbb F_q[u],
\end{equation}
is the restriction of the Fourier transform of $t$ to the perpendicular space to $\{ f \in \mathbb F_q[u], |f|<X\}$. A geometric analogue of this argument lets us compute that the $\ell$-adic Fourier transform of the natural sheaf whose trace function is the one in \cref{FunctionSelectIntervalShortSum} from the Fourier transforms of the sheaves giving rise to $t$. From this, and the Katz--Laumon result that Fourier transforms preserve (semi)perversity, we get that this sheaf is (semi)perverse, which gives strong control of the stalk at the generic point, which is the compactly supported cohomology of the generic member of the family. 

To control vanishing cycles, we isolate a crucial property called translation-invariance. Our sheaves $\mathcal F_\pi$ are locally isomorphic to their pullbacks by translation $\mathbb A^1 \to \mathbb A^1$ everywhere away from the finite singularities. At finite points away from the singularities, the sheaf is lisse (roughly, locally constant), since by definition the singularities are where the sheaf is not lisse. More remarkably, the same property is true locally near the point $\infty$ in the usual compactification $\mathbb P^1$ of $\mathbb A^1$. This is the (only) place where our assumption that the slopes of $\mathcal F_\pi$ at $\infty$ are bounded above by $1$ is used - we show that this bound on the slopes is equivalent to translation-invariance (at $\infty$).

The sheaf we must calculate the vanishing cycles of is constructed from the individual sheaves $\mathcal F_\pi$. The vanishing cycles measures how much this sheaf varies as we move the interval around. Since moving the interval around is a form of translation, only the sheaves $\mathcal F_\pi$ which are singular at a given point matter. When only a few sheaves matter, the function $t$ behaves like a function on $\mathbb F_q[u]/(\widetilde g)$ for $\widetilde g$ some proper divisor of $g$. As long as $|\widetilde g|$ is less than $X$, the sum in \cref{FunctionSelectIntervalShortSum}  is independent of $h$, suggesting that the cohomology does not vary and thus that there are no vanishing cycles. This is indeed what happens, namely vanishing cycles occur only at points where many sheaves $\mathcal F_\pi$ are singular. These points form a closed set of high codimension, giving a strong vanishing result for compactly supported \'etale cohomology.

\section{Translation-Invariant Sheaves on $\mathbb A^1$} \label{TransInvSection}

Here we introduce and study a key local property of sheaves on the affine line called translation-invariance.

Let $\mathcal F$ be a sheaf on $\mathbb A^1$, let $\upsilon \colon \mathbb A^1 \to \mathbb P^1$ be the usual open immersion, and let $\textup{trans} \colon \mathbb P^1 \times \mathbb A^1 \to \mathbb P^1$ be the unique map whose restriction to $\mathbb A^1 \times \mathbb A^1 \subset \mathbb A^1 \times \mathbb P^1$ is the addition map $\mathbb A^1 \times \mathbb A^1 \to \mathbb A^1$. That is
\[\textup{trans} ( (a_0: a_1) , b) = (a_0: a_1 + b a_0 ).\]
Let $\pi_1 \colon \mathbb P^1 \times \mathbb A^1 \to \mathbb P^1 $ be the projection onto the first factor.

\begin{defi} We say $\mathcal F$ is translation-invariant at a point $x \in \mathbb P^1$ if there exists an isomorphism between 
$\pi_1^* \upsilon_! \mathcal F$ and $\textup{trans}^* \upsilon_! \mathcal F$ after pullback to a suitable \'etale neighborhood of $(x,0) \in \mathbb P^1 \times \mathbb A^1.$ \end{defi}

\begin{lem} \label{AffineTranslationInvariance}

For $x \in \mathbb A^1 \subset \mathbb P^1$, the sheaf $\mathcal F$ is translation-invariant at $x$ if and only if $\mathcal F$ is lisse at $x$. 

\end{lem}

\begin{proof} 

For the `if' direction, because $\mathcal F$ is lisse at $x \in \mathbb A^1$, the sheaf $\upsilon_! \mathcal F$ is lisse at $x \in \mathbb P^1$. Take an \'{e}tale neighborhood $U$ of $x$ over which $\upsilon_! \mathcal F$ is constant. Then the fiber product over $\mathbb P^1 \times \mathbb A^1$ of the inverse images $\pi_1^{-1}(U)$ and $\textup{trans}^{-1}( U)$ is an \'etale neighborhood of $(x,0)$ over which the pullbacks of $\pi_1^* \upsilon_! \mathcal F$ and $\textup{trans}^* \upsilon_! \mathcal F$ are constant sheaves of the same rank, hence isomorphic to each other.

For the `only if' direction, take an \'etale neighborhood $U$ of $(x,0)$ over which $\pi_1^* \upsilon_! \mathcal F$ and $\textup{trans}^* \upsilon_! \mathcal F$ become isomorphic. Restrict this neighborhood $U$ to an \'etale neighborhood $U'$ of $(x,0) \in \{x\} \times \mathbb A^1$. Abusing notation, we see that $\pi_1 \colon \{x\} \times \mathbb A^1 \to \mathbb P^1$ is a constant morphism, so the restriction of $\pi_1^* \upsilon_! \mathcal F$ to $U'$ is a constant sheaf. It follows that the restriction of $\textup{trans}^* \upsilon_! \mathcal F$ to $U'$ is constant as well. 
In terms of the map 
\[
i_x \colon \mathbb A^1 \to \mathbb P^1 \times \mathbb A^1, \quad i_x(y) = (x,y)
\]
this means that $i_x^* \textup{trans}^* \upsilon_! \mathcal F$ is locally constant at $0$. 
Let $[+x] \colon \mathbb A^1 \to \mathbb A^1$ be the map given by 
\[
[+x](y) = x + y 
\]
and note that $\textup{trans} \circ i_x = v \circ [+x]$ so 
\[
i_x^* \textup{trans}^* \upsilon_! \mathcal F = [+x]^* \upsilon^* \upsilon_! \mathcal F = [+x]^* \mathcal F. 
\]
We conclude that $[+x]^* \mathcal F$ is locally constant at $0$, 
 so $\mathcal F$ is locally constant at $x$, as desired.
\end{proof}

\begin{thm} \label{TranslationInvarianceOnTheLine}

A sheaf $\mathcal F$ on $\mathbb A^1$ is translation-invariant at $\infty \in \mathbb P^1$ if and only if all the slopes of the local monodromy representation of $\mathcal F$ at $\infty$ are bounded from above by $1$.

\end{thm}

\begin{remark}

It would be interesting to obtain an analog of \cref{TranslationInvarianceOnTheLine} in the realm of ordinary differential equations over $\C$.

\end{remark}

\begin{lem} \label{TensorProdTransInvLine}

Let $\mathcal F$ and $\mathcal G$ be sheaves on $\mathbb A^1$ that are translation-invariant at $x \in \mathbb P^1$. 
Then so is $\mathcal F \otimes \mathcal G$.

\end{lem}

\begin{proof}

We just need to note that tensor product of sheaves commutes with extension by zero and with pullback.
\end{proof}

Another way to prove \cref{TensorProdTransInvLine} is to use \cref{AffineTranslationInvariance} in conjunction with the fact that the tensor product of two sheaves lisse at a certain point is also lisse at that point, and \cref{TranslationInvarianceOnTheLine} in conjunction with \cite[Proposition 2.11, Proof of Corollary 2.12]{SS20}.

\begin{lem} \label{DualTransInvLine}

Let $\mathcal F$ be a sheaf on $\mathbb A^1$ that is translation invariant at some $x \in \mathbb P^1$.
Then so is $\mathcal F^{\vee}$.

\end{lem}

\begin{proof}

We just need to note that duality of sheaves commutes with extension by zero and with pullback.
\end{proof}

An alternative way to prove \cref{DualTransInvLine} is to use \cref{AffineTranslationInvariance} in conjunction with the fact that the dual of a sheaf lisse at a certain point is also lisse at that point, and \cref{TranslationInvarianceOnTheLine} in conjunction with the fact that the slopes of a representation coincide with the slopes of its dual.

\begin{lem} \label{NegationTransInvLine}

Let $N \colon \mathbb A^1 \to \mathbb A^1$ be the morphism mapping $x$ to $-x$, and let $\mathcal F$ be a sheaf on $\mathbb A^1$ that is translation-invariant at $\infty$.
Then $N^* \mathcal F$ is also translation-invariant at $\infty$.

\end{lem}

\begin{proof}

Follows immediately from the fact that $N$ extends to an automorphism of $\mathbb P^1$, intertwines translation maps, and commutes with $\upsilon_!$.
\end{proof}

Another way to prove \cref{NegationTransInvLine} is to use \cref{TranslationInvarianceOnTheLine} in conjunction with the fact that $N$ maps a uniformizer at $\infty$ to another such uniformizer, and thus preserves the slopes of any representation.

\begin{lem} \label{FourierTransformTranslationInvariantLine}

Let $\mathcal F$ be a sheaf on $\mathbb A^1$ having no finitely supported sections and neither constant nor Artin--Schreier factors in its global monodromy.
Then $\mathcal F$ is translation-invariant at $\infty$ if and only if its Fourier transform is.

\end{lem}

\begin{proof}

In view of \cite[Theorem 2 (1)]{Trav} and our assumptions on $\mathcal F$, the Fourier transform of $\mathcal F$ is a sheaf on $\mathbb A^1$.
By the involutivity of the Fourier transform, it suffices to prove one direction.
This follows from \cref{TranslationInvarianceOnTheLine} and \cite[Proposition 2.4.3 (i)(b), (iii)(b)]{Lau}.
\end{proof}

Let $\mathcal C$ be the category of sheaves on $\mathbb A^1$ that are lisse on $\mathbb G_m = \mathbb A^1 \setminus \{0\}$, tamely ramified at $0$ with vanishing stalk, and totally wild at infinity.
\cite{Katz88} endows $\mathcal C$ with an operation of convolution by pulling back to $\mathbb G_m$, convolving, and extending by zero to $\mathbb A^1$.

\begin{lem}

The subcategory of those sheaves in $\mathcal C$ that are translation-invariant at $\infty$ is preserved under convolution. 

\end{lem}

\begin{proof}

Follows from \cref{TranslationInvarianceOnTheLine}, \cite[Section 7.6]{Katz88}, and \cite[Proposition 2]{RL}.
\end{proof}

\begin{cor}

Let $\mathcal F$ and $\mathcal G$ be sheaves on $\mathbb A^1$ that are translation-invariant at $\infty$, have no finitely supported sections, and have neither constant nor Artin--Schreier factors in their global monodromy.
Then their additive convolution $\mathcal F * \mathcal G$ is a sheaf on $\mathbb A^1$ that is translation-invariant at $\infty$. 

\end{cor}

\begin{proof}

This follows from \cite[Proposition 1.2.2.7]{Lau}, \cref{FourierTransformTranslationInvariantLine}, \cref{NegationTransInvLine}, \cref{DualTransInvLine}, and the fact that negation and duality preserve the lack of finitely supported sections and Artin--Schreier factors in the global monodromy.
\end{proof}

Our next goal is to prove \cref{TranslationInvarianceOnTheLine}.

Let $F$ be the field of fractions of a Henselian discrete valuation ring $\mathcal{O}_F$ whose residue field is perfect and has prime characteristic $p$.
Let $K/F$ be a finite totally ramified Galois extension with Galois group $G$. The field $K$ is Henselian as well, and we denote its valuation ring by $\mathcal{O}_K$.
Let
\[
e = \abs{G} = [K : F]
\]
and let $l \geq -1$ be the largest integer for which the $l$th ramification group $G_l$ in the lower numbering filtration is nontrivial.
For a real number $z \geq -1$ we set $G_z = G_{\lceil z \rceil}$. 
Our assumption that $K/F$ is totally ramified means that $G_0 = G$.

We also consider the ramification groups in the upper numbering filtration characterized by the property $G^r = G_i$ where
\[ r = \int_{0}^i  \frac{ \abs{G_z}}{\abs{G_0}} dz \geq -1 \] 
so it makes sense to define
\begin{equation*} 
s = \max \ \{r \geq -1 : G^r \neq 1\}.
\end{equation*}
We recall that the Artin conductor of a finite-dimensional complex representation $V$ of $G$ is given by
\[ c_V = \sum_{i=0}^{l} \frac{ \abs{G_i}}{ \abs{G_0}}  ( \dim_{\mathbb C} V - \dim_{\mathbb{C}} V^{G_i} ) \]
as in \cite[p. 100, Cor. 1']{Ser}.

The following is a variant of \cite[p. 103, Ex. 2]{Ser}.

\begin{lem} \label{pre-valuation-lemma}

For the regular representation $\operatorname{Reg}$ of $G$ we have 
\[
c_{\operatorname{Reg}} = \sum_{i=0}^{l} ( \abs{G_i} - 1)
\]
and $e(s + 1) -1   = c_{\operatorname{Reg}} +l.$
In particular, $s \in \mathbb{Q}$ and $es \in \mathbb{Z}$.

\end{lem}

\begin{proof}

On the one hand, by applying the Mackey decomposition and Frobenius reciprocity to the restriction of $\operatorname{Reg}$ to $G_i$ we get that
\[c_{\operatorname{Reg}}  = \sum_{i=0}^{l} \frac{ \abs{G_i}}{ \abs{G}} \left( \abs{G} - \frac{\abs{G} }{ \abs{G_i}} \right) = \sum_{i=0}^{l} ( \abs{G_i} - 1) .\] 

On the other hand, because of the maximality of $s$ and of $l$, we must have
\[ s=  \int_{0}^l  \frac{ \abs{G_z}}{\abs{G_0}} dz \] 
and $l$ is an integer so
\[s =  \sum_{i=1}^{l} \int_{i-1}^i  \frac{ \abs{G_z}}{\abs{G_0}} dz = \sum_{i=1}^{l}  \frac{ \abs{G_i}}{\abs{G_0}} =  \frac{1}{e} \sum_{i=1}^{l} \abs{G_i}. \]
We therefore get
\[ es + e -1 = e-1+ \sum_{i=1}^{l} \abs{G_i} = -1 + \sum_{i=0}^{l} \abs{G_i}  = l + \sum_{i=0}^{l} ( \abs{G_i} -1) = l + c_{\operatorname{Reg}} .\]
\end{proof}

The tame quotient $T = G/G_1$ is a finite abelian group of order not divisible by $p$.
We denote this order by $\tau$.

\begin{lem} \label{CongruenceArtinConductor}

We have
$
c_{\operatorname{Reg}} \equiv -1 \mod \tau.
$

\end{lem}

\begin{proof}

Consider the action of the group $T^\vee$ (identified with those homomorphisms from $G$ to $\mathbb{C}^\times$ that are trivial on $G_1$) on the set $\mathcal{R}$ of (isomorphism classes of) finite-dimensional complex irreducible representations of $G$ by taking tensor products. 
The subset $\mathcal{R}^{\text{tame}}$ of $\mathcal{R}$ of those representations that factor via $T$ forms a single orbit of length $\tau$. The Artin conductor of each nontrivial representation in this orbit is $1$, while the Artin conductor of the trivial representation is $0$, so
\begin{equation*}
c_{\text{Reg}} = \sum_{V \in \mathcal{R}} c_V \dim_{\mathbb{C}} V = 
\tau - 1 + \sum_{V \in \mathcal{R} \setminus \mathcal{R}^{\text{tame}}} c_V \dim_{\mathbb{C}} V.
\end{equation*}

Our task is therefore to show that $\tau$ divides the last sum above.
Fixing a representation $V$ of $G$ that is not tame, it is sufficient to show that $\tau$ divides
\[
\sum_{W \in O_V} c_W \dim_{\mathbb{C}} W
\]
where $O_V$ is the orbit of $V$ under the aforementioned action of $T^\vee$.
The representations in $O_V$ are tensor products of $V$ by one-dimensional representations so
\[
\sum_{W \in O_V} c_W \dim_{\mathbb{C}} W = \dim_{\mathbb{C}} V \sum_{W \in O_V} c_W.
\]

We claim that $c_W = c_V$ for every $W \in O_V$. To see this, note first that for every $1 \leq i \leq l$, the $G_i$-representations $V$ and $W$ are isomorphic. Second we note that
\[
\dim_{\mathbb{C}} W^{G} = \dim_{\mathbb{C}} (V \otimes \chi)^{G} = \dim_{\mathbb{C}} \mathrm{Hom}_G(\bar{\chi},V) = 0
\]
where $\chi \in T^\vee$, and the last equality follows from Schur's Lemma and our assumption that the irreducible representation $V$ is not tame, 
hence not isomorphic to the irreducible representation $\bar{\chi}$. Our claim follows.

Denoting by $S_V \leq T^\vee$ the stabilizer of $V$, we get
\[
\dim_{\mathbb{C}} V \sum_{W \in O_V} c_W = c_V \dim_{\mathbb{C}} V |O_V| = c_V \tau \frac{\dim_{\mathbb{C}} V }{|S_V|}.
\]
It would therefore be enough to show that $|S_V|$ divides $\dim_{\mathbb{C}} V$.
To do this we consider the subgroup
\[
H = \bigcap_{\chi \in S_V} \mathrm{Ker}(\chi)
\]
of $T$, and let $H_0$ be the inverse image of $H$ under the quotient map $G \to T$. We have
\[
[G : H_0] = [T : H] = |(T/H)^\vee| = |\{\chi \in T^\vee : H \subseteq \mathrm{Ker}(\chi) \}| = |S_V|
\]
so we are tasked with showing that $[G : H_0]$ divides $\dim_{\mathbb{C}} V$.

Denoting by $\chi_V$ the character of $V$, we claim that $\chi_V(g) = 0$ for every $g \in G$ that does not lie in $H_0$.
Indeed, if that was not the case, we could find a character $\psi \in T^\vee$ with 
\[
\psi(g) \neq 1, \quad H \subseteq \mathrm{Ker}(\psi)
\]
and get that $V \otimes \psi \ncong V$ since their characters satisfy $\chi_V(g) \psi(g) \neq \chi_V(g)$. This would mean that $\psi \notin S_V$, contradicting our assumption that $H$ is contained in $\mathrm{Ker}(\psi)$ because
\[
S_V = \{\chi \in T^\vee : H \subseteq \mathrm{Ker}(\chi)\}.
\]
The claim is thus proven.

Since $V$ is an irreducible representation of $G$ we have
\[
[G : H_0] = \frac{1}{|H_0|} \sum_{g \in G} |\chi_V(g)|^2 = \frac{1}{|H_0|} \sum_{g \in H_0} |\chi_V(g)|^2
\]
which is the number of (not necessarily distinct) irreducible representations of $H_0$ appearing in a direct sum decomposition of the restriction of $V$ to $H_0$. 
Since $H_0$ is a normal subgroup of $G$, Clifford theory tells us that these irreducible representations of $H_0$ all have the same dimension, say $d$, 
so $[G : H_0]d = \dim_{\mathbb{C}} V$. 
In particular, $[G : H_0]$ divides $\dim_{\mathbb{C}} V$, which is what we had to demonstrate.
\end{proof}

\begin{cor} \label{IntegralityScor}

The following three conditions are equivalent.

\begin{itemize}

\item The integer $\tau$ divides $l$.

\item The rational number $s$ satisfies $|G_1|s \in \Z$.

\item The rational number $s$ lies in $\Z[p^{-1}]$.

\end{itemize}

\end{cor}

\begin{proof}

By \cref{pre-valuation-lemma} we have $l = es + e - (c_{\operatorname{Reg}}+1)$, where $e$ is divisible by $\tau$ in view of Lagrange's theorem and $c_{\operatorname{Reg}}+1$ is divisible by $\tau$ in view of \cref{CongruenceArtinConductor}. We conclude that $\tau$ divides $l$ if and only if
\[
|G_1|s = \frac{e}{\tau}s \in \mathbb{Z} 
\]
so the first two conditions are equivalent, and they imply the third condition because $|G_1|$ is a power of $p$.

To show that the third condition implies the second, we take a power $q$ of $p$ with $qs \in \Z$. 
By \cref{pre-valuation-lemma}, we also have $es \in \Z$ so $\gcd(q,e)s \in \Z$. Since $\tau$ is not divisible by $p$ we have
\[
\gcd(q,e) = \gcd(q, \tau |G_1|) = \gcd(q, |G_1|)
\]
which divides $|G_1|$, hence $|G_1|s \in \Z$. Our three conditions are thus seen to be equivalent.
\end{proof}

Let $t$ be a uniformizer of $F$, let $u$ be a uniformizer of $K$, and let $v$ be the ($u$-adic) valuation on $K$.  
Our assumption that $K/F$ is totally ramified means that $\mathcal{O}_K = \mathcal{O}_F[u]$ and that the the minimal polynomial 
\[
f(x) = x^e + \sum_{j=0}^{e-1} a_jx^j \in \mathcal{O}_F[x]
\]
of $u$ over $F$ is an Eisenstein polynomial. In particular $t$ divides $a_j$ for every $0 \leq j \leq e-1$.

Next we show that the Newton polygon of $f(u+x)$ lies above a certain line.

\begin{lem}\label{valuation-lemma}

We write 
\[
f(u+x) = x^e + \sum_{j=1}^{e-1} b_j x^j
\]
with $b_j \in \mathcal{O}_K$. Then, for every $1 \leq j \leq e-1$ we have 
\begin{equation}\label{u-valuation-bound} v( b_j) \geq e(s+1) - j (l+1) .\end{equation}
We have equality in \eqref{u-valuation-bound} if $j=1$. If $j > 1$ and $s>0$, then we have equality in \eqref{u-valuation-bound} only if $j$ is a power of $p$ and $j \leq \abs{G_l}$. 

\end{lem}

\begin{proof}

The Galois conjugates of $u$ are the roots of $f$, so we have
\[ f(u+x) = \prod_{ \sigma \in G}  ( u+ x- \sigma(u)) \]
and thus
\[ b_j = \sum_{ \substack{S \subseteq G \\ \abs{S}=j}} \prod_{ \sigma \in G \setminus S} (u - \sigma(u))\]
for every $1 \leq j \leq e-1$.
There is no contribution from subets $S$ not containing $1$  because for such $S$ there is a factor corresponding to $\sigma = 1$ which vanishes.
Hence, 
\[v (b_j) \geq \min_{ \substack{S \subseteq G \\ \abs{S}=j \\ 1 \in S}} \sum_{ \sigma \in G \setminus S} v( u-\sigma(u)) \geq \sum_{ \sigma \in G\setminus \{1\}} v(u - \sigma(u) ) -  (j-1) \max_{ \sigma \in G \setminus \{1\}}  v( u-\sigma(u)).\]

By definition of $G_i$, for each $\sigma \in G \setminus \{1\}$ we have 
\[v(u-\sigma(u)) = \max \{ i \geq 0 \mid \sigma \in G_{i-1} \}  = \sum_{i=1}^{l+1} \mathbf{1}_{ \sigma \in G_{i-1}} \] 
so 
\[  \max_{ \sigma \in G \setminus \{1\}}  v( u-\sigma(u)) = \max \{ i \geq 0 \mid  \sigma \in G_{i-1} \textrm{ for some }\sigma \neq 1 \} = l + 1\] and 
\[ \sum_{ \sigma \in G\setminus \{1\}} v(u - \sigma(u) ) = \sum_{ \sigma \in G\setminus \{1\}} \sum_{i=1}^{l+1} 1_{ \sigma \in G_{i-1}} =  \sum_{ \sigma \in G\setminus \{1\}} \sum_{i=0}^{l} \mathbf{1}_{ \sigma \in G_{i}}
 = \sum_{i=0}^{l} (\abs{G_i} - 1 ) = c\]
 by \cref{pre-valuation-lemma}.
Thus \[ v(b_j) \geq c - (j-1) (l+1) = c + l+1 - j (l+1)  = e (s+1) - j(l+1) \]
where the last equality again follows from \cref{pre-valuation-lemma}.
 
The $u$-adic valuation of $b_j$ is exactly $ e (s+1) - j(l+1) $ if and only if 
\[ \sum_{ \substack{S \subseteq G \\ \abs{S}=j \\ 1 \in S}} \prod_{ \sigma \in G \setminus S} (u - \sigma(u)) \not \equiv  0 \mod u^{  e (s+1) + 1 - j (l+1) } \]
in $\mathcal{O}_K$. 
For every $j$-element subset $S$ of $G$ containing $1$ we have 
\begin{equation*}
\begin{split} 
v \left( \prod_{ \sigma \in G \setminus S} (u - \sigma(u))  \right)  &= \sum_{ \sigma \in G\setminus \{1\}} v(u - \sigma(u) ) -  \sum_{\sigma \in S\setminus\{1\}} v( u-\sigma(u)) \\
&\geq   \sum_{ \sigma \in G\setminus \{1\}} v(u - \sigma(u) ) -  (j-1) \max_{ \sigma \in G \setminus \{1\}}  v( u-\sigma(u))
\end{split}
\end{equation*}
which we have already seen equals $e (s+1) - j (l+1)$.
The inequality above is an equality if (and only if) $ v( u-\sigma(u)) = l+1$ for all $\sigma \in S \setminus \{1\}$, or equivalently $S \subseteq G_l$. 
 
 So if $S \not\subseteq G_l$ then 
 \[
 \prod_{ \sigma \in G \setminus S} (u - \sigma(u)) \equiv 0 \mod u^{ e(s+1) + 1 - j (l+1)}
 \] 
 hence 
 \[ \sum_{ \substack{S \subseteq G \\ \abs{S}=j \\ 1 \in S}} \prod_{ \sigma \in G \setminus S} (u - \sigma(u))  \equiv   \sum_{ \substack{S \subseteq G_l \\ \abs{S}=j \\ 1 \in S}} \prod_{ \sigma \in G \setminus S} (u - \sigma(u))  \mod u^{  e (s+1) + 1 - j (l+1) }. \]
For $\sigma \in G_l \setminus \{1\}$, we can write  $u - \sigma(u) = \epsilon_\sigma u^{l + 1} + \rho_\sigma$ such that $v(\rho_\sigma) \geq l+2$ 
and $\epsilon_\sigma \in \mathcal{O}_K^\times$.
It follows that \[\prod_{ \sigma \in G \setminus S} (u - \sigma(u)) \equiv \prod_{ \sigma \in G_l  \setminus S} \epsilon_\sigma u^{l+1} \prod_{\sigma \in G \setminus G_l} (u-\sigma(u))   \mod u^{ e(s+1) + 1 - j (l+1)}\]  
so the $u$-adic valuation of $b_j$ is exactly $e(s+1) - j (l+1)$ if and only if
\[\sum_{ \substack{ S \subseteq G_l \\ \abs{S}= j \\ 1 \in S}} \prod_{ \sigma \in G_l \setminus S} \overline{\epsilon_\sigma} \neq 0 \]
where $\overline{\epsilon_\sigma}$ is the image of $\epsilon_\sigma$ in the residue field $k = \mathcal{O}_K/(u)$ of $K$.
 
In case $j=1$ there is only one summand above, corresponding to $S = \{1\}$, so we get a product of elements in $k^\times$. 
In case $j > 1$ and $s > 0$ (equivalently $l > 0$), the group $G_l$ is an elementary abelian $p$-group because $\sigma \mapsto \overline{\epsilon_\sigma}$ becomes an injective group homomorphism form $G_l$ to $k$ once we set $\overline{\epsilon_1} = 0$.
It follows that the polynomial
\[
\prod_{ \sigma \in G_l } (x + \overline{\epsilon_\sigma}) \in k[x]
\]
is additive, so its $j$th coefficient, which is precisely the sum we are interested in, does not vanish only if $j$ is a power of $p$ not exceeding $\abs{G_l}$.
\end{proof}

Let $R$ be a strictly Henselian local ring with residue field $k$.

\begin{prop} \label{HenselianFactors}

Let $\alpha \colon \mathcal{O}_F \to R$ be a homomorphism of local rings. 
Let
\[
\alpha(f) = x^e + \sum_{j=0}^{e-1} \alpha(a_j)x^j \in R[x]
\]
be the polynomial obtained from $f$ by applying $\alpha$ to each coefficient of $f$.
Then $R[x]/(\alpha(f))$ is a (strictly) Henselian local ring with residue field $k$, and maximal ideal containing $x$.


\end{prop}

\begin{proof}

The ring $R[x]/(\alpha(f))$ is finite over $R$, so it is a finite product of Henselian local rings.
We need to show that $R[x]/(\alpha(f))$ has a unique maximal ideal, that $x$ lies in this ideal, and that the quotient of $R[x]/(\alpha(f))$ by this ideal is the quotient of $R$ by its maximal ideal. 

We first take a maximal ideal $M$ of $R[x]/(\alpha(f))$, and show that $x \in M$.
By definition of $f$, we have the equality
\[
x^e = - \sum_{j=0}^{e-1} \alpha(a_j) x^j
\]
in $R[x]/(\alpha(f))$.
Since $M$ is a maximal ideal, in particular a radical ideal, in order to conclude that $x \in M$ it suffices to show that
\[
\sum_{j=0}^{e-1} \alpha(a_j) x^j \in M.
\]
It would therefore be enough to prove that $\alpha(a_j) \in M$ for every $0 \leq j \leq e-1$.

Since $\alpha(a_j) \in R$, we need to show that $\alpha(a_j) \in M \cap R$.
As $R[x]/(\alpha(f))$ is an integral extension of $R$, the ideal $M \cap R$ is maximal in $R$, and $R$ is local, so $M \cap R$ is the unique maximal ideal of $R$.
We have assumed that $a_j$ lies in the maximal ideal of $\mathcal{O}_F$, and that $\alpha$ is a local homomorphism, so $\alpha(a_j)$ lies in the maximal ideal of $R$, as required.

We conclude that the natural map 
\[
\Spec R[x]/(\alpha(f),x) \to \Spec R[x]/(\alpha(f))
\]
induces a bijection on maximal ideals. We have
\[
R[x]/(\alpha(f),x) \cong R/(\alpha(a_0))
\]
and we have seen that $\alpha(a_0)$ lies in the maximal ideal of $R$ so $R/(\alpha(a_0))$ is a local ring as well.
It follows that $R[x]/(\alpha(f))$ has only one maximal ideal, and that the quotient by this ideal is naturally identified with the residue field of $R/(\alpha(a_0))$. The latter field is $k$, the residue field of $R$.
\end{proof}

\begin{lem}\label{galois-agreement-lemma}  

Let $\alpha_1, \alpha_2 \colon \mathcal O_F \to R$ be two homomorphisms of local rings. Suppose that, for each $x \in \mathcal O_F$,   \[ \alpha_1(x) \equiv \alpha_2(x) \mod  \alpha_1(t)^{ 1 + \lceil s \rceil} \] and, if $s=0$, that $\alpha_1(t)$ and $\alpha_2(t)$ differ by multiplication by a unit.
Then \[ R \otimes_{ \mathcal O_F, \alpha_1} \mathcal O_K \cong R \otimes_{\mathcal O_F, \alpha_2} \mathcal O_K\] 
as $R$-algebras with an action of $G$. If $R$ is an integral domain and $\alpha_1(t) \neq 0$, then there are only finitely many distinct isomorphisms between the two $R$-algebras above.

\end{lem}

\begin{proof}

%

We identify $\mathcal O_K$ with $\mathcal O_F[x]/(f(x))$ as $\mathcal{O}_F$-algebras. 
Let $\alpha_1(f)$ (respectively, $\alpha_2(f)$) be the polynomial in $R[x]$ (respectively, $R[y]$) obtained from $f(x) \in R[x]$ (respectively, $f(y) \in R[y]$) by applying $\alpha_1$ (respectively, $\alpha_2$) to each coefficient. 
Then 
\[ R \otimes_{ \mathcal O_F, \alpha_1} \mathcal{O}_K \cong R [x]/ (\alpha_1(f) ), \quad R \otimes_{ \mathcal O_F, \alpha_2} \mathcal{O}_K \cong R [y]/ (\alpha_2(f) ) \] 
as $R$-algebras with an action of $G$, so it suffices to show that
\[ R[x]/ (\alpha_1(f)) \cong  R[y]/(\alpha_2(f)).\]

We will construct an isomorphism by mapping $y$ to an appropriately chosen root of $\alpha_2(f)$ in $R[x]/(\alpha_1(f))$.
Suppose now that $R$ is an integral domain with field of fractions $\tilde{R}$, and that $\alpha_1(t) \neq 0$. 
We show that $\alpha_2(f)$ has only finitely many roots in $R[x]/(\alpha_1(f))$.
Since $f$ is monic, the natural homomorphism of $R$-algebras $R[x]/(\alpha_1(f)) \to \tilde{R}[x]/(\alpha_1(f))$ is, on the level of $R$-modules, 
the natural inclusion of a free $R$-module of rank $e$ into an $e$-dimensional vector space over $\tilde{R}$, so it suffices to show that $\alpha_2(f)$ has only finitely many roots in $\tilde{R}[x]/(\alpha_1(f))$.

Our assumptions allows us to (uniquely) extend $\alpha_1$ to a homomorphism of fields $\beta_1 \colon F \to \tilde{R}$, 
so the $\tilde{R}$-algebra $\tilde{R}[x]/(\alpha_1(f))$ is the base change from $F$ to $\tilde{R}$ of the finite \'etale algebra $K$.
It follows that $\tilde{R}[x]/(\alpha_1(f))$ is a finite \'etale algebra over $\tilde{R}$, namely it is isomorphic to a finite direct product of (finite) field extensions of $\tilde{R}$.
Each such extension contains only finitely many elements where $\alpha_2(f)$ vanishes, so $\alpha_2(f)$ has only finitely many roots in $\tilde{R}[x]/(\alpha_1(f))$.
In particular, there are only finitely many isomorphism of $R$-algebras from $R \otimes_{ \mathcal O_F, \alpha_1} \mathcal O_K$ to $R \otimes_{ \mathcal O_F, \alpha_2} \mathcal O_K$.

If $s=0$ then $p$ does not divide $e$ and $\mathcal{O}_K \cong \mathcal{O}_F[x]/(x^e - t)$ as $\mathcal{O}_F$-algebras.
In this case $G$ is cyclic, and if we pick a generator $g$ of $G$, we get that $g(x) = \zeta_e x$ for a primitive root of unity $\zeta_e \in \mathcal{O}_F$ of order $e$.
Therefore
\[
R \otimes_{ \mathcal O_F, \alpha_1} \mathcal O_K \cong R[x]/(x^e - \alpha_1(t)), \quad R \otimes_{ \mathcal O_F, \alpha_2} \mathcal O_K \cong R[y]/(y^e - \alpha_2(t))
\]
as $R$-algebras with a $G$-action. 
By assumption there exists $\chi \in R^\times$ such that $\alpha_2(t) = \chi \cdot \alpha_1(t)$.
Using our assumption that $R$ is strictly Henselian, we can find $z \in R^\times$ such that $z^e = \chi$.
The $R$-algebra homomorphism induced by $y \mapsto zx$ gives the required $G$-equivariant isomorphism.

Assume from now on that $s>0$. In the polynomial ring $R[y]$ we have
\[
\alpha_2(f)(y) \equiv \alpha_1(f)(y) \mod \alpha_1(t)^{ \lceil s \rceil+1}.
\]
Since $u^e$ divides $t$ in $\mathcal{O}_K$, we get that $x^e$ divides $\alpha_1(t)$ in $R[x]/(\alpha_1(f))$, 
so from the congruence above we deduce that
\[
\alpha_2(f)(y) \equiv \alpha_1(f)(y) \mod x^{ e(\lceil s \rceil+1)}.
\]
in $\left( R[x]/(\alpha_1(f)) \right)[y]$.
\cref{pre-valuation-lemma} tells us that $e(s+1)$ is an integer (or equivalently that $s \in \frac{1}{e}\Z$), so the congruence above implies
\begin{equation} \label{PolynomialsThirdAlphaCongruence}
\begin{cases}
\alpha_2(f)(y) \equiv \alpha_1(f)(y) \mod x^{ e(s+1)} &s \in \Z \\
\alpha_2(f)(y) \equiv \alpha_1(f)(y) \mod x^{ e(s+1) + 1} &s \notin \Z.
\end{cases}
\end{equation}

It follows from \cref{valuation-lemma} that if we write 
\begin{equation*} \label{PrimeBjToDivisibility}
\alpha_1(f) (x+ y) = y^e + \sum_{j=1}^{e-1} b_j' y^j \in \big( R[x]/(\alpha_1(f)) \big) [y],
\end{equation*}
then $b_j'x^{j(l+1)}$ is divisible by $x^{ e(s+1) }$, it is divisible by $x^{ e(s + 1) + 1}$ unless $j$ is $1$ or a power of $p$, 
and it is $x^{ e(s + 1)}$ times a unit of $R[x]/(\alpha_1(f))$ if $j=1$. 
It follows from \cref{PolynomialsThirdAlphaCongruence} and \cref{HenselianFactors} that if we write 
\begin{equation*} 
\alpha_2(f) (x+ y) = \sum_{j=0}^{e} c_j y^j \in \big( R[x]/(\alpha_1(f)) \big) [y],
\end{equation*}
then the coefficients $c_0, \dots, c_e$ have the following properties.
\begin{itemize}

\item The element $c_j x^{j(l+1)}$ is divisible by $x^{e(s+1)}$ for every $0 \leq j \leq e$.

\item If $s \notin \Z$ then $c_0$ is divisible by $x^{e(s+1) + 1}$.

\item The element $c_j x^{j(l+1)}$ is divisible by $x^{e(s+1) + 1}$ for every $2 \leq j \leq e$ that is not a power of $p$.

\item There exists $\xi \in \big( R[x]/(\alpha_1(f)) \big)^\times$ such that $c_1 = \xi x^{e(s+1) - (l+1)}$.

\end{itemize}

We infer that the coefficient of $y$ in the polynomial 
\begin{equation} \label{eq-divided-polynomial} 
E(y) = \frac{ \alpha_2(f) ( x+ x^{l+1} y) }{  x^{e(s+1)} } \in \big( R[x]/(\alpha_1(f)) \big)[y]
\end{equation} 
is a unit, and for each $2 \leq j \leq e$ that is not a power of $p$, the coefficient of $y^j$ is divisible by $x$. 
If $s \notin \Z$, then the constant term of $E(y)$ is divisible by $x$ as well.
By \cref{HenselianFactors}, $R[x]/(\alpha_1(f))$ is a strictly Henselian local ring with maximal ideal containing $x$ and residue field $k$.
By assumption, the characteristic of the residue field $\mathcal{O}_F/(t)$ of $F$ is $p$, and $\alpha_1$ induces an injection of this residue field into $k$, 
so the characteristic of $k$ is $p$ as well.
As a result, the reduction $\overline E \in k[y]$ of $E$ is a polynomial with a nonzero constant derivative.
If $s \notin \Z$ we have $E(0) \equiv 0 \ \mathrm{mod} \ x$ and thus $\overline{E}(0) = 0$.
Since $R[x]/(\alpha_1(f))$ is Henselian, there is a unique root $r \in R[x]/(\alpha_1(f))$ of $E$ that reduces to $0 \in k$.
As $R[x]/(x, \alpha_1(f))$ is Henselian as well, we conclude that 
\begin{equation} \label{NonintegralScongruenceR}
r \equiv 0 \mod x.
\end{equation}

If $s \in \Z$, we note that $\overline E$ is a nonconstant separable polynomial over the separably closed field $k$, so it has a root $z \in k$.
As in the previous case, there is a unique root $r \in R[x]/(\alpha_1(f))$ of $E$ that reduces to $z \in k$.

%

We claim that the homomorphism of $R$-algebras
\[ i \colon  R [y]/ (\alpha_2(f) ) \to R [x]/ (\alpha_1(f) )\]
that sends $y$ to $x + x^{l+1} r$ is an isomorphism. 
Since both $R$-algebras are free $R$-modules of rank $e$, it suffices to check that the induced $k$-linear map 
\[
\overline{i} \colon  k[y]/ (\overline{\alpha_2(f)} ) \to k[x]/ ( \overline{\alpha_1(f)} )
\]
is an isomorphism.
Since $f$ is an Eisenstein polynomial, we have 
\[
\overline{\alpha_1(f)}(x) = x^e, \quad \overline{\alpha_2(f)}(y) = y^e.
\]
Therefore, in the basis $1, y, \dots, y^{e-1}$ for the source and the basis $1, x, \dots, x^{e-1}$ for the target, 
the map $\overline{i}$ is represented by an upper-triangular matrix with diagonal entries $1$, so $\overline{i}$ is indeed an isomorphism.

To check that $i$ is $G$-equivariant, it suffices to check for each $\sigma \in G$ that $\sigma^{-1} \circ i \circ \sigma  = i  $. Since $R[y]/(\alpha_2(f))$ is generated over $R$ by $y$, it is enough to check that 
\begin{equation} \label{TheEquivarianceOfIunderG}
\sigma^{-1} (i( \sigma (y) )) = i(y).
\end{equation}  
Since $E(r) = 0$ we see that $i(y)$ is a root of $\alpha_2(f)$.

We claim that $\sigma^{-1}(i( \sigma(y) ))$ is also a root of $\alpha_2(f)$.
By definition, $G$ acts trivially on $R$ so this amounts to showing that
\begin{equation} \label{VanishingSeconRootSigma}
\alpha_2(f)(i( \sigma(y) )) = 0
\end{equation}
in $R[x]/(\alpha_1(f))$.
There exists a polynomial $P$ with coefficients in $\mathcal{O}_F$ such that $\sigma(y) = \alpha_2(P)$ and $f$ divides $f \circ P$.
We get
\[
\alpha_2(f)(i( \sigma(y) )) = \alpha_2(f)(i( \alpha_2(P) )) = \alpha_2(f)(\alpha_2(P)(i(y))) = \alpha_2(f \circ P)(i(y))
\]
so \cref{VanishingSeconRootSigma} follows from $i(y)$ being a root of $\alpha_2(f)$ and $f$ dividing $f \circ P$.

The homomorphism of $R$-algebars
\[
\widetilde{i} \colon R[y]/(\alpha_2(f),y^{l+1}) \to R[x]/(\alpha_1(f), x^{l+1})
\]
induced by $i$ sends $y$ to $x$ so
\[
i(y) \equiv x \mod x^{l+1}, \quad \sigma^{-1} (i( \sigma (y) )) \equiv x \mod x^{l+1}.
\]
Therefore in $R[x]/(\alpha_1(f))$ we have the two roots
\[
E\left( \frac{i(y)-x}{x^{l+1}} \right) = 0, \quad E\left( \frac{\sigma^{-1}(i(\sigma(y)))-x}{x^{l+1}} \right) = 0.
\]
Since $R[x]/(\alpha_1(f))$ is Henselian, in order to see that these roots are equal and arrive at \cref{TheEquivarianceOfIunderG},
it suffices to check that the roots are equal mod $x$ or equivalently 
\begin{equation} \label{LplusTwoCongruence}
\sigma^{-1}(i( \sigma (y) )) \equiv  i(y) \mod x^{l+2}.
\end{equation}

If $s \notin \Z$ then \cref{LplusTwoCongruence} follows immediately from \cref{NonintegralScongruenceR}, so the proof in this case is complete.
We assume henceforth that $s \in \Z$.

Suppose that $\sigma \in G_1$. In this case there exists a polynomial $w$ with coefficients  in $\mathcal{O}_F$ such that
\[
\sigma(x) = x + \alpha_1(w)x^2, \quad \sigma(y) = y + \alpha_2(w)y^2
\]
in $R[x]/(\alpha_1(f))$ and $R[y]/(\alpha_2(f))$ respectively.
We therefore have
\begin{equation*}
\begin{split}
i(\sigma(y)) &= i(y + \alpha_2(w)y^2) = x + x^{l+1}r + i(\alpha_2(w))(x + x^{l+1}r)^2 \\
&\equiv x + x^{l+1}r + \alpha_2(w)x^2 \mod x^{l+2}
\end{split}
\end{equation*}
and
\begin{equation*}
\begin{split}
\sigma(i(y)) &= \sigma(x + x^{l+1}r) = x + \alpha_1(w)x^2 + (x + \alpha_1(w)x^2)^{l+1} \sigma(r) \\
&\equiv x + \alpha_1(w)x^2 + x^{l+1}\sigma(r) \mod x^{l+2}.
\end{split}
\end{equation*}

As $G$ acts trivially on $R$ we have $\sigma(r) \equiv r \ \mathrm{mod} \ x$ so 
\[
x^{l+1}r \equiv x^{l+1}\sigma(r) \mod x^{l+2}.
\]
Hence, in order to show that \cref{LplusTwoCongruence} holds, it suffices to verify that for each $\beta \in \mathcal{O}_F$ we have 
$
\alpha_1(\beta) = \alpha_2(\beta)
$
in $R[x]/(\alpha_1(f),x^{l})$. This follows from the argument leading to \cref{PolynomialsThirdAlphaCongruence}, and the fact that $l \leq e(s+1)$ proven in \cref{pre-valuation-lemma}.

Suppose now that $\sigma \notin G_1$. In this case there exists $\zeta \in R^\times$ satisfying $\zeta^\tau = 1$ such that
\[
\sigma(x) \equiv \zeta x + \alpha_1(\omega)x^2, \quad \sigma(y) = \zeta y + \alpha_2(\omega) y^2
\]
for some polynomial $\omega$ over $\mathcal{O}_F$. We therefore have
\[
i(\sigma(y)) \equiv \zeta x + \zeta x^{l+1}r + \alpha_2(\omega) x^2 \mod x^{l+2}
\]
and
\[
\sigma(i(y)) \equiv \zeta x + \alpha_1(\omega)x^2 + \zeta^{l+1} x^{l+1} \sigma(r) \mod x^{l+2}.
\]
On top of the arguments used in the previous case, here we need the additional fact that $\zeta^l = 1$ which follows from \cref{IntegralityScor} as the latter tells us that $\tau \mid l$ because $s \in \mathbb{Z}$.
This concludes the proof of $G$-equivariance of $i$ in the case $s \in \Z$, and thus the entire argument. 
\end{proof}


\begin{cor} \label{SummaryHenselianCor}

Let $F$ be the field of fractions of a Henselian discrete valuation ring $\mathcal{O}_F$ whose residue field is perfect and has prime characteristic $p$.
Let $K/F$ be a totally ramified Galois extension with Galois group $G$, let $\mathcal{O}_K$ be the valuation ring of $K$, and set
\begin{equation*} 
s = \sup \ \{r \geq -1 : G^r \neq 1\}, \quad G^r = \varprojlim_{U \lhd_o G} (G/U)^r.
\end{equation*}

Let $R$ be a strictly Henselian integral domain and let $\alpha_1, \alpha_2 \colon \mathcal O_F \to R$ be two homomorphisms of local rings. Let $t \in \mathcal{O}_F$ be a uniformizer, suppose that $\alpha_1(t) \neq 0$, and suppose that for each $x \in \mathcal O_F$ we have   
\[ \alpha_1(x) \equiv \alpha_2(x) \mod  \alpha_1(t)^{ 1 + \lceil s \rceil}. \]
If $s=0$, suppose also that $\alpha_1(t)$ and $\alpha_2(t)$ differ by multiplication by a unit.
Then $R \otimes_{ \mathcal O_F, \alpha_1} \mathcal O_K \cong R \otimes_{\mathcal O_F, \alpha_2} \mathcal O_K$ 
as $R$-algebras with an action of $G$.

\end{cor}

\begin{proof}

For each open normal subgroup $U$ of $G$ denote by $K^U$ its fixed subfield of $K$ and by $\mathcal{O}_{K^U}$ the valuation ring of $K^U$.
By \cref{galois-agreement-lemma}, the set $\mathcal{I}_U$
of $G/U$-equivariant isomorphisms of $R$-algebras from $R \otimes_{ \mathcal O_F, \alpha_1} \mathcal O_{K^U}$ to $R \otimes_{ \mathcal O_F, \alpha_2} \mathcal O_{K^U}$ is finite and nonempty. 

If $V$ is an open normal subgroup of $G$ contained in $U$, 
we claim that restriction of isomorphisms to $R \otimes_{ \mathcal O_F, \alpha_1} \mathcal O_{K^U}$ defines a function from $\mathcal{I}_V$ to $\mathcal{I}_U$. Indeed, let 
\[
\varphi \colon R \otimes_{ \mathcal O_F, \alpha_1} \mathcal O_{K^V} \to 
R \otimes_{ \mathcal O_F, \alpha_2} \mathcal O_{K^V}
\]
be an isomorphism from $\mathcal{I}_V$. Since $\varphi$ is $G/V$-equivariant, it restricts to a $G/U$-equivariant isomorphism
\[
\varphi^{U/V} \colon \left( R \otimes_{ \mathcal O_F, \alpha_1} \mathcal O_{K^V} \right)^{U/V} \to 
\left( R \otimes_{ \mathcal O_F, \alpha_2} \mathcal O_{K^V} \right)^{U/V}
\] 
of $R$-algebras. 
Since $\mathcal{O}_{K^V}$ is a free module of finite rank over $\mathcal{O}_F$, the isomorphism above is (naturally identified with) the isomorphism
\[
\varphi^{U/V} \colon R \otimes_{ \mathcal O_F, \alpha_1} \mathcal O_{K^U} \to 
R \otimes_{ \mathcal O_F, \alpha_2} \mathcal O_{K^U}.
\]
That is, the restriction $\varphi^{U/V}$ of $\varphi$ to $R \otimes_{ \mathcal O_F, \alpha_1} \mathcal O_{K^U}$ lies in $\mathcal{I}_U$, as claimed.

We see that the $\mathcal{I}_U$, with the restriction functions above, form an inverse system (over the directed set of open normal subgroups of $G$) of nonempty finite sets. The inverse limit of this system is nonempty, and every member of this inverse limit gives rise to a $G$-equivariant isomorphism of $R$-algebras from $R \otimes_{ \mathcal O_F, \alpha_1} \mathcal O_K$ to $R \otimes_{ \mathcal O_F, \alpha_2} \mathcal O_K$.  
\end{proof}

\begin{proof}[Proof of \cref{TranslationInvarianceOnTheLine}] 

Suppose that all the slopes of the monodromy representation of $\mathcal F$ at $\infty$ are at most $1$.
We need to show that the sheaves $\pi_1^* \upsilon_! \mathcal F$ and $\textup{trans}^* \upsilon_! \mathcal F$ are isomorphic in some \'etale neighborhood of $(\infty,0)$ in $\mathbb P^1 \times \mathbb A^1$. It suffices to show that these sheaves become isomorphic after pullback to the spectrum of the strict henselization $R$ of the local ring of $\mathbb P^1 \times \mathbb A^1$ at $(\infty, 0)$, namely
\[
R = \mathcal{O}_{\mathbb{P}^1 \times \mathbb A^1, (\infty,0)}^{\text{sh}} = \mathcal{O}_{\mathbb{P}^1 \times \mathbb A^1, \overline{(\infty, 0)}}. 
\] 
In other words, $R$ is the \'etale local ring of $\mathbb P^1 \times \mathbb A^1$ at a geometric point over $(\infty, 0)$.  
We take $t$ to be the inverse of the standard coordinate on $\mathbb P^1$, namely we write $\mathbb{A}^1 = \Spec \overline{\F_q}[t^{-1}]$.

We have the commutative diagram
\[
\begin{tikzcd}
\Spec R[t^{-1}] \arrow[dd, "\beta_2"', bend left = 52] \arrow[r] \arrow[dd, "\beta_1", bend right = 52] \arrow[rrd, "\widetilde{\pi_1}" description, bend left = 7] \arrow[rrd, "\widetilde{\textup{trans}}" description, bend right = 7] & \Spec R \arrow[dd, "\Spec \alpha_2" near end, bend left = 55] \arrow[r] \arrow[dd, "\Spec \alpha_1"', bend right = 55]  & \mathbb{P}^1 \times \mathbb A^1 \arrow[d, dotted, "\textup{trans}", bend left = 15] \arrow[d, dotted, "\pi_1"', bend right = 15] \\ 
& & \mathbb{A}^1 \arrow[d,"v"] \\ 
\Spec \mathcal{O}_{\mathbb{P}^1, \overline{\infty}}[t^{-1}] \arrow[r] \arrow[rur, "\zeta"] & \Spec \mathcal{O}_{\mathbb{P}^1, \overline{\infty}} \arrow[r] & \mathbb{P}^1
\end{tikzcd}
\]
where $\alpha_1, \beta_1, \widetilde{\pi_1}$ correspond to $\pi_1$ and $\alpha_2, \beta_2, \widetilde{\textup{trans}}$ correspond to $\textup{trans}$.
Since pullback is compatible with extension by zero, it suffices to show that on $\operatorname{Spec} R[t^{-1}]$ we have
\[
\widetilde{\pi_1}^* \mathcal{F} \cong \widetilde{\textup{trans}}^* \mathcal{F}. 
\]
Commutativity of the diagram above gives
\[
\widetilde{\pi_1}^* \mathcal{F} \cong \beta_1^* \zeta^* \mathcal{F}, \quad \widetilde{\textup{trans}}^{*} \mathcal{F} \cong \beta_2^* \zeta^* \mathcal{F}.
\]

We denote by $F$ the (valued) field $\mathcal{O}_{\mathbb{P}^1, \overline{\infty}}[t^{-1}]$ whose (discrete) valuation ring is the Henselian local ring $\mathcal{O}_F = \mathcal{O}_{\mathbb{P}^1, \overline{\infty}}$. The sheaf $\zeta^* \mathcal{F}$ is a continuous finite-dimensional representation of the absolute Galois group $\Gamma$ of $F$ over $\overline{\mathbb{Q}_\ell}$. We let $N$ be the closed normal subgroup of $\Gamma$ consisting of all those elements of $\Gamma$ that act trivially on our representation,
and take $K$ to be the subfield of some separable closure of $F$ containing all the elements that are fixed by $N$. 
We also set $G = \Gamma/N = \mathrm{Gal}(K/F)$, and note that $K/F$ is totally ramified since the residue field $\overline{\F_q}$ of $F$ is algebraically closed.
In terms of the ramification groups in the upper numbering filtration on $G$, our assumption that all the slopes of $\mathcal{F}$ at $\infty$ are at most $1$ means that
\begin{equation*} \label{BoundForSlopeInImage}
\sup \ \{r \geq -1 : G^r \neq 1\} \leq 1.
\end{equation*}

Denoting by $y$ the coordinate on the second copy of $\mathbb A^1$, we see that 
\[
\alpha_1(t) = t, \quad \alpha_2(t) = \frac{t}{1+yt}
\]
so $\alpha_1(t)$ and $\alpha_2(t)$ differ by multiplication by the unit $1 + yt$ of $R$, and
\[
\alpha_1(t) - \alpha_2(t) = t - \frac{t}{1+yt} = \frac{yt^2}{1+yt} \equiv 0 \mod t^2.
\]
Since the $\overline{\F_q}$-algebra 
$
\mathcal{O}_{\mathbb{P}^1, \overline{\infty}}/(t^2) = \overline{\F_q}[t]/(t^2)
$
is generated by $t$, in the commutative diagram
\begin{equation*} 
\begin{tikzcd}  
\mathcal{O}_{\mathbb{P}^1, \overline{\infty}} \arrow[r, ""]  \arrow[d,"\alpha_1"', bend right] \arrow[d,"\alpha_2", bend left] & \mathcal{O}_{\mathbb{P}^1, \overline{\infty}}/(t^2)  \arrow[d,"\overline{\alpha_1}"', bend right] \arrow[d,"\overline{\alpha_2}", bend left]  \\
R \arrow[r, ""]  & R/(t^2) 
\end{tikzcd}
\end{equation*}
we have $\overline{\alpha_1} = \overline{\alpha_2}$, so for every $x \in \mathcal{O}_{\mathbb{P}^1, \overline{\infty}}$ we get that
$
\alpha_1(x) \equiv \alpha_2(x) \ \mathrm{mod} \ t^2.
$

Since $R$ is an integral domain, we get from \cref{SummaryHenselianCor} that
\[ R \otimes_{ \mathcal O_F, \alpha_1} \mathcal O_K \cong R \otimes_{\mathcal O_F, \alpha_2} \mathcal O_K\]
as $R$-algebras with an action of $G$. From this we conclude that
\[
R[\lambda^{-1}] \otimes_{ F, \beta_1} K \cong R[\lambda^{-1}] \otimes_{F, \beta_2} K
\]
as $R[\lambda^{-1}]$-algebras with a $G$-action, so we have the pair of Cartesian squares
\begin{equation*} 
\begin{tikzcd}  
\Spec R[\lambda^{-1}]  \otimes_{F} K  \arrow[r,"\overline \epsilon"]  \arrow[d,"\beta_1^K"', bend right] \arrow[d,"\beta_2^K", bend left]  & \Spec R[\lambda^{-1}]  \arrow[d,"\beta_1"', bend right] \arrow[d,"\beta_2", bend left]  \\
\Spec K \arrow[r, "\epsilon"]  & \Spec F
\end{tikzcd}
\end{equation*}
having the horizontal maps $\epsilon$ and $\overline \epsilon$ in common.

In order to deduce that
\[
\beta_1^* \zeta^* \mathcal{F} \cong \beta_2^* \zeta^* \mathcal{F}
\]
as sheaves on $\Spec R[\lambda^{-1}]$, it is therefore enough to show that 
\[
\overline{\epsilon}^* \beta_1^* \zeta^* \mathcal{F} \cong \overline{\epsilon}^* \beta_2^* \zeta^* \mathcal{F}
\]
as constant sheaves on $\Spec R[\lambda^{-1}]  \otimes_{F} K $ (equivalently, vector spaces over $\overline{\mathbb{Q}_\ell}$) with an action of $G$. 
This follows from the fact that pulling back the constant sheaf $\epsilon^* \zeta^* \mathcal{F}$ (with its $G$-action) by $\beta_1^K$ is the same thing as pulling it back by $\beta_2^K$.
The proof of one implication is thus complete.

Suppose now that $\mathcal F$ is translation-invariant at $\infty$. \cite[Definition 2.4.2.3]{Lau} introduces a functor $\mathscr F_\psi^{\infty, \infty}$ from the category of representations of the absolute Galois group $\Gamma$ of the local field $F$ to itself. From the definition of this functor, one can see that it intertwines shifting by $a \in \overline{\F_q}$ with tensoring by the Artin--Schreier sheaf $\mathcal L_\psi( a t^{-1})$, namely
\[ \mathscr F_\psi^{\infty, \infty} (  [ t^{-1} \mapsto t^{-1} + a]^* \mathcal F) = \mathscr F_\psi^{\infty, \infty} ( \mathcal F) \otimes \mathcal L_\psi(a t^{-1}).\]
To see this, we state in our notation Laumon's definition of $\mathscr F_\psi^{\infty, \infty}$. Let $j \colon \operatorname{Spec} F \to \operatorname{Spec} \mathcal O_F$ be the open immersion, and $\textup{pr} \colon \operatorname{Spec} \mathcal O_F \otimes_{\overline{\F_q}} \mathcal O_F \to  \operatorname{Spec} \mathcal O_F$ the projection onto the first factor. Let $t_1$ and $t_2$ be the variable $t$ in $\mathcal O_F$ viewed as an element of $ \mathcal O_F \otimes_{\overline{\F_q}} \mathcal O_F$ by the map coming respectively from the first and second tensor power.   Let $u\colon \operatorname{Spec} F \otimes_{\overline{\F_q}} F \to \operatorname{Spec} \mathcal O_F \otimes_{\overline{\F_q}} \mathcal O_F$ be the natural open immersion.

Then \cite[Definition 2.4.2.3]{Lau} puts
\[  \mathscr F_\psi^{\infty, \infty} (  \mathcal F)= R^1 \Phi \left(  \textup{pr}^*  j_! \mathcal F \otimes  u_! \mathcal L_\psi\left( \frac{1}{t_1t_2} \right) \right)\]
where we take vanishing cycles on $ \operatorname{Spec} \mathcal O_F \otimes_{\overline{\F_q}} \mathcal O_F$ with respect to the map to $\operatorname{Spec} \mathcal O_F$ via the projection onto the second factor.

Now we have
\begin{equation*}
\begin{split}
\mathscr F_\psi^{\infty, \infty} ( [t^{-1} \mapsto t^{-1}+a]^* \mathcal F) &= R^1 \Phi \left(  \textup{pr}^*  j_! [t^{-1} \mapsto t^{-1}+a]^* \mathcal F \otimes  u_! \mathcal L_\psi \left(\frac{1}{t_1t_2}  \right) \right) \\
&= R^1 \Phi \left(  \textup{pr}^* \left[t \mapsto \frac{t}{1+at} \right]^*  j_!  \mathcal F \otimes  u_! \mathcal L_\psi \left( \frac{1}{t_1t_2}   \right) \right)
\\
&= R^1 \Phi \left(  \left[t_1 \mapsto \frac{t_1}{1+at_1}, t_2 \mapsto t_2 \right]^* \textup{pr}^*   j_!  \mathcal F \otimes  u_! \mathcal L_\psi \left( \frac{1}{t_1t_2}  \right) \right) \\
&= R^1 \Phi \left(  \left[t_1 \mapsto \frac{t_1}{1+at_1}, t_2 \mapsto t_2 \right]^* \left(  \textup{pr}^*   j_!  \mathcal F \otimes  [t_1 \mapsto \frac{t_1}{1+at_1}, t_2 \mapsto t_2 ]_* u_! \mathcal L_\psi \left( \frac{1}{t_1t_2}   \right) \right) \right) \\
&= R^1 \Phi \left( \textup{pr}^*   j_!  \mathcal F \otimes  \left[t_1 \mapsto \frac{t_1}{1+at_1}, t_2 \mapsto t_2 \right]_* u_! \mathcal L_\psi \left( \frac{1}{t_1t_2}  \right) \right) \\
&= R^1 \Phi \left(\textup{pr}^*   j_!  \mathcal F \otimes  u_!  [t_1^{-1}  \mapsto t_1^{-1} +a_1 , t_2 \mapsto t_2 ]_* \mathcal L_\psi \left(\frac{1}{t_1t_2} \right) \right) \\
&= R^1 \Phi \left( \textup{pr}^*   j_!  \mathcal F \otimes  u_!  \mathcal L_\psi \left( \left( \frac{1}{t_1}+a \right) \frac{1}{t_2}     \right) \right) \\
&= R^1 \Phi \left(\textup{pr}^*   j_!  \mathcal F \otimes  u_!  \mathcal L_\psi  \left( \frac{1}{t_1 t_2}  + \frac{a}{t_2}    \right)  \right) \\
&= R^1 \Phi \left(\textup{pr}^*   j_!  \mathcal F \otimes  u_!  \mathcal L_\psi \left(  \frac{1}{t_1 t_2}    \right) \otimes \mathcal L_\psi \left( \frac{a}{t_2} \right) \right) \\
&=  R^1 \Phi \left(\textup{pr}^*   j_!  \mathcal F \otimes  u_!  \mathcal L_\psi \left( \frac{1}{t_1 t_2}    \right)  \right) \otimes \mathcal L_\psi \left( \frac{a}{t} \right)
=   \mathscr F_\psi^{\infty, \infty} (  \mathcal F) \otimes \mathcal L_\psi \left( \frac{a}{t} \right)
\end{split} 
\end{equation*}
where we use crucially two properties of vanishing cycles: in the fourth equality, we use that the vanishing cycles of a sheaf on  $ \operatorname{Spec} \mathcal O_F \otimes_{\overline{\F_q}} \mathcal O_F$ with respect to the map to $\operatorname{Spec} \mathcal O_F$ are preserved by pulling back the sheaf under an automorphism of $ \operatorname{Spec} \mathcal O_F \otimes_{\overline{\F_q}} \mathcal O_F$  that fixes the projection to $\operatorname{Spec} \mathcal O_F$, and in the ninth equality, we use that taking the tensor product of a sheaf on $ \operatorname{Spec} \mathcal O_F \otimes_{\overline{\F_q}} \mathcal O_F$ whose special fiber is zero with the pullback of a lisse sheaf on the generic point of  $\operatorname{Spec} \mathcal O_F$ has the effect on the vanishing cycles of a sheaf with respect to the map to  $\operatorname{Spec} \mathcal O_F$ of tensoring with that same lisse sheaf. Both properties are immediate from the definition of vanishing cycles. 

Therefore our assumption that $\mathcal F$ is invariant under translation by $a$ for all $a \in \overline{\F_q}$ implies that $\mathscr F_\psi^{\infty,\infty}(\mathcal F)$ is invariant under tensor product with $\mathcal L_\psi( \frac{a}{t} )$ for all $a \in \overline{\F_q}$. However, a nonzero $d$-dimensional representation of a group may only be invariant under tensor product by at most $d^2$ one-dimensional representations of the group, since if $V \otimes \mathcal L \cong V$ then $\mathcal L$ is a summand of $V \otimes V^\vee$ unless $V=0$ and there can be at most $\dim (V \otimes V^\vee)=d^2$ such summands. Since all the Artin--Schreier sheaves give distinct representations of $\Gamma$, and  $\mathscr F_\psi^{\infty,\infty}(\mathcal F)$ is finite-dimensional, it follows that $\mathscr F_\psi^{\infty,\infty}(\mathcal F)=0$.

However, by \cite[Theorem 2.4.3.(iii).a) and c)]{Lau}, the functor $\mathscr F_\psi^{\infty,\infty}$ is exact, and, restricted to the category of representations $\Gamma$ with slope $>1$, is an equivalence of categories from the category of representations with slope $>1$ to itself. Hence if the local mondoromy representation of $\mathcal F$ contains a nontrivial subrepresentation with slope $>1$, $\mathscr F_\psi^{\infty,\infty}(\mathcal F)$ contain $\mathscr F_\psi^{\infty,\infty}$ of that subrepresentation, which is nonzero since an equivalence of categories cannot send a nonzero object to a zero object. Since $\mathscr F_\psi^{\infty,\infty}(\mathcal F)=0$, all slopes of $\mathcal F$ at infinity must be $\leq 1$, as desired.\end{proof}

\section{Vanishing of Cohomology} \label{VanishingOfCohomologySection}


Let $m$ be a positive integer and let $g \in \overline{\F_q}[u]$ be a squarefree polynomial of degree $m$.
We view $\mathbb A^m$ as parametrizing polynomials in a single variable $u$ over $\overline {\F_q}$ of degree less than $m$ via their coefficients.
We can also think of $\mathbb A^m$ as parametrizing residue classes modulo $g$.
Let $\overline{e}_i \colon \mathbb A^m \to \mathbb A^{1}$ for $1 \leq i \leq m$ be the map that evaluates a polynomial at $x_i$, and for constructible \'etale $\overline{\mathbb Q_\ell}$-sheaves $\mathcal F_1, \dots, \mathcal F_m$ on $\mathbb A^1$ set
\[
\overline{\mathcal F} = \bigotimes_{i=1}^m \overline{e}_i^* \mathcal F_i.
\]
Let $\pi \colon \mathbb A^m \to \mathbb A^{m-n}$ denote the projection onto the coefficients of $u^{n}, \dots, u^{m-1}$, so that $\pi^{-1}(0)$ is the space $\mathbb{A}^n$ of polynomials of degree less than $n$. We therefore view $\mathbb{A}^{m-n}$ as the space of polynomials of degree less than $m$ modulo the subspace of polynomials of degree less than $n$, or as a space parametrizing intervals modulo $g$ of a given length.  
Let $e_i$ be the restriction of $\overline{e}_i$ to $\mathbb A^n = \pi^{-1}(0)$, and let the sheaf
\[
\mathcal F = \bigotimes_{i=1}^m e_i^* \mathcal F_i.
\]
be the pullback of the sheaf $\overline{\mathcal F}$ under the inclusion of $\mathbb A^n = \pi^{-1}(0)$ into $\mathbb A^m$.
Our goal is to show in \cref{TechHeart} that, under suitable assumptions on $\mathcal F_1$, \dots, $\mathcal F_m$, the compactly supported \'etale cohomology of $\mathcal F$ vanishes in degrees $j > n+1$.

From the proper base change theorem over the Cartesian square
\begin{equation} \label{SmallFiberCartesianSquare}
\begin{tikzcd}  
\mathbb A^{n} \arrow[d,""] \arrow[r, hook, ""] & \mathbb A^{m} \arrow[d, "\pi"] \\
\operatorname{Spec} \overline{\F_q} \arrow[r, "0"] & \mathbb A^{m-n} 
\end{tikzcd}
\end{equation}
we get that  $H^j_c(\mathbb A^n, \mathcal F)$ is isomorphic to the stalk of $R^j \pi_! \overline{\mathcal F}$ at (a geometric point over) $0 \in \mathbb{A}^{m-n}$. 
In order to prove \cref{TechHeart} it is therefore sufficient to show that this stalk vanishes for every $j > n + 1$.

\begin{lem}\label{generic-step} 

Let $\overline{\xi}$ be a geometric generic point of $\mathbb A^{m-n}$. 
For every $1 \leq i \leq m$ suppose that $\mathcal F_i$ has no finitely supported sections, and that its (geometric) global monodromy has neither Artin--Schreier nor trivial factors.
Then for every integer $j > n$ we have 
\[
(R^j \pi_! \overline{\mathcal F})_{\overline{\xi}} = 0.
\]

\end{lem}

\begin{proof} 

By the definition of semiperversity, it suffices to show that $  R\pi_! \overline{\mathcal F} [m]$ is semiperverse. 
As an operation on $D_c^b(\mathbb{A}^{m-n}, \overline{\mathbb{Q}_\ell})$, the Fourier transform $\mathrm{FT}_\psi$ preserves semiperversity, 
and $\mathrm{FT}_\psi \circ \mathrm{FT}_{\overline \psi}$ is the identity, so it suffices to show that  
$ \mathrm{FT}_\psi  R\pi_! \overline{\mathcal F} [m]$ is semiperverse.

Taking $\text{pr}_1, \text{pr}_2 \colon \mathbb A^{m-n} \times \mathbb A^{m-n} \to \mathbb A^{m-n}$ to be the projection maps, and $b \colon \mathbb A^{m-n} \times \mathbb A^{m-n} \to \mathbb A^1$ to be a nondegenerate bilinear form, we get from the definition of the Fourier transform that
\[ \mathrm{FT}_\psi  R\pi_! \overline{\mathcal F} [m] =  
R\text{pr}_{2!}\left( \mathcal L_\psi(b)\otimes  \text{pr}_1^* R\pi_! \overline{\mathcal F} \right) [2m-n]     .\]

Let $\text{pr}_1' \colon \mathbb{A}^m \times \mathbb{A}^{m-n} \to \mathbb{A}^m$ and $\text{pr}_2' \colon \mathbb{A}^m \times \mathbb{A}^{m-n} \to \mathbb{A}^{m-n}$ be the projections. Setting $\pi' = \pi \times \mathrm{id}$ we get the Cartesian square
\[ \begin{tikzcd}  \mathbb A^m \times \mathbb A^{m-n} \arrow[d,"\text{pr}_1'"] \arrow[r, "\pi'"] & \mathbb A^{m-n} \times \mathbb A^{m-n}  \arrow[d, "\text{pr}_1"] \\
\mathbb A^m \arrow[r, "\pi"] & \mathbb A^{m-n} \end{tikzcd}\]
so from proper base change and the projection formula we get
\begin{equation*}
\begin{split}
R\text{pr}_{2!}\left( \mathcal L_\psi(b)\otimes  \text{pr}_1^* R\pi_! \overline{\mathcal F} \right) [2m-n] &= 
R\text{pr}_{2!}  \left(  \mathcal L_\psi(b) \otimes  R\pi'_{!} \text{pr}_1^{'*}  \overline{\mathcal F} \right)[2m-n] \\ &=
R \text{pr}_{2!}   R \pi'_{!} \left( \mathcal L_\psi( b \circ \pi')   \otimes \text{pr}_1^{'*} \overline{\mathcal F}  \right) [2m-n].
\end{split}
\end{equation*}

Now we can view the second $\mathbb A^{m-n}$ as the space of polynomials over $\overline{\F_q}$ of degree less than $m-n$, at which point we can choose our bilinear form $b$ to be multiplication of polynomials modulo $g$ followed by extracting the coefficient of $u^{m-1}$, since this is a nondegenerate bilinear form between 
polynomials of degree less than $m$ modulo polynomials of degree less than $n$ and polynomials of degree less than $m-n$.
Equivalently, using the residue theorem, we can write
\[ b \circ \pi' ( a_0,\dots, a_{m-1} ;  \tilde{a}_0,\dots, \tilde{a}_{m-n-1} ) \] 
as
\begin{equation*}
\begin{split}
&\operatorname{Res}_{\infty}  \frac{ \left( a_0 + a_1 u + \dots + a_{m-1} u^{m-1}\right)  \left( \tilde{a}_0 + \tilde{a}_1 u + \dots + \tilde{a}_{m-n-1} u^{m -n-1}\right ) }{ g} = \\ 
&- \sum_{i=1}^m  \operatorname{Res}_{x_i}   \frac{ \left( a_0 + a_1 u + \dots + a_{m-1} u^{m-1}\right)  \left( \tilde{a}_0 + \tilde{a}_1 u + \dots + \tilde{a}_{m-n-1} u^{m -n-1}\right)  }{ g}. 
\end{split}
\end{equation*}

We conclude that
\begin{equation*}
\begin{split}
\mathcal L_\psi( b \circ \pi')  &= \bigotimes_{i=1}^m \mathcal L_\psi \left(  \operatorname{Res}_{x_i}   \frac{ \left( a_0 + a_1 u + \dots + a_{m-1} u^{m-1}\right)  \left( \tilde{a}_0 + \tilde{a}_1 u + \dots + \tilde{a}_{m-n-1} u^{m -n-1}\right)  }{ g}\right) \\
&= \bigotimes_{i=1}^m  (\overline{e}_i \times \tilde{e}_i)^* \mathcal L_\psi (b_i)  
\end{split}
\end{equation*} 
where $\tilde{e}_i \colon \mathbb A^{m-n} \to \mathbb A^{1}$ is the map evaluating at $x_i$, and $b_i$ is the bilinear form defined on scalars by
\[
b_i(\alpha, \beta) = \operatorname{Res}_{x_i} \frac{ \alpha \beta}{ g} = \frac{\alpha \beta}{g'(x_i)}.
\] 
Therefore
\begin{equation*}
\begin{split}
R \text{pr}_{2!}   R \pi'_{!} \left( \mathcal L_\psi( b \circ \pi')   \otimes \text{pr}_1^{'*} \overline{\mathcal F}  \right) [2m-n]  &= 
R(\text{pr}_2 \circ \pi'  )_{!} \bigotimes_{i=1}^m  \left(    (\overline{e}_i \times \tilde{e}_i)^* \mathcal L_\psi (b_i)  \otimes \text{pr}_{1}^{'*}  \overline{e}_i^* \mathcal F_i \right) [2m-n]  \\
&=R\text{pr}'_{2!} \bigotimes_{i=1}^m  \left(    (\overline{e}_i \times \tilde{e}_i)^* \mathcal L_\psi (b_i)  \otimes (\overline{e}_i \circ \text{pr}_1')^* \mathcal F_i \right) [2m-n].
\end{split}
\end{equation*}

Let $\widetilde{\text{pr}}_{1} \colon  \mathbb A^{1} \times \mathbb A^{m-n} \to  \mathbb A^{1}$ and $\widetilde{\text{pr}}_{2} \colon \mathbb A^{1} \times \mathbb A^{m-n} \to \mathbb A^{m-n}$ be the projections. We can view $\text{pr}_2'$ as the $m$th fibered power over $\mathbb{A}^{m-n}$ of the projection $\widetilde{\text{pr}}_{2} $, because each polynomial in $\mathbb A^m$ is uniquely determined by its values at $x_1, \dots, x_m$ by Lagrange's interpolation.
Since 
$
\overline{e}_i \circ \text{pr}_1' =  \widetilde{\text{pr}}_{1} \circ (\overline{e}_i \times \mathrm{id}),
$
by applying the $m$-fold relative K\"unneth formula, we obtain
\begin{equation*}
R\text{pr}'_{2!} \bigotimes_{i=1}^m  \left(    (\overline{e}_i \times \tilde{e}_i)^* \mathcal L_\psi (b_i)  \otimes (\overline{e}_i \circ \text{pr}_1')^* \mathcal F_i \right) [2m-n] = 
\bigotimes_{i=1}^m R \widetilde{\text{pr}}_{2!} \left(   (\mathrm{id} \times \tilde{e}_i)^* \mathcal L_\psi ( b_i)  \otimes  \widetilde{\text{pr}}_{1}^* \mathcal F_i \right) [2m-n].
\end{equation*}

Let $q_{1}, q_{2} \colon \mathbb A^{1} \times \mathbb A^{1} \to \mathbb A^{1}$ be the projections. These give rise to the Cartesian square
\[ \begin{tikzcd}  
\mathbb A^{1} \times \mathbb A^{m-n} \arrow[d,"\widetilde{\text{pr}}_{2}"] \arrow[r, "\mathrm{id} \times \tilde{e}_i"] & \mathbb A^{1} \times \mathbb A^{1}  \arrow[d, "q_{2}"] \\
\mathbb A^{m-n} \arrow[r, "\tilde{e}_i"] & \mathbb A^{1} \end{tikzcd}\]
so the fact that $q_{1} \circ (\mathrm{id} \times \tilde{e}_i) = \widetilde{\text{pr}}_{1}$ and proper base change give
\begin{equation*}
\begin{split}
&\bigotimes_{i=1}^m R\widetilde{\text{pr}}_{2!} \left(   (\mathrm{id} \times \tilde{e}_i)^* \mathcal L_\psi ( b_i )  \otimes  \widetilde{\text{pr}}_{1}^* \mathcal F_i \right) [2m-n] = \\ 
&\bigotimes_{i=1}^m R\widetilde{\text{pr}}_{2!} \left(   (\mathrm{id} \times \tilde{e}_i)^* \mathcal L_\psi ( b_i)  \otimes  (\mathrm{id} \times \tilde{e}_i)^* q_{1}^* \mathcal F_i \right) [2m-n] = 
\bigotimes_{i=1}^m  \tilde{e}_i^*   Rq_{2!} (  \mathcal L_\psi(b_i) \otimes q_{1}^* \mathcal F_i )  [2m-n].
\end{split}
\end{equation*}

From the definition of the Fourier transform on $\mathbb A^{1}$ with the character $\psi$, taking $b_i$ as our (nondegenerate) bilinear form, we see that
\[ \bigotimes_{i=1}^m  \tilde{e}_i^*  R q_{2!} (  \mathcal L_\psi(b_i) \otimes q_{1}^* \mathcal F_i )  [2m-n] = \bigotimes_{i=1}^m  \tilde{e}_i^* \mathrm{FT}_{ \psi , b_i}  \mathcal F_i   [m-n]. \]
For every $1 \leq i \leq m$, it follows from our initial assumptions on $\mathcal F_i$ and \cite[Theorem 2 (1)]{Trav} that the Fourier transform $\mathrm{FT}_{ \psi  ,b_i  }  \mathcal F_i  $ is a sheaf on $\mathbb A^{1}$. 
Indeed, it does not matter which nondegenerate bilinear form $b_i$ we choose as all are equivalent up to scaling of $\mathbb A^{1}$.
It follows that $\bigotimes_{i=1}^m  \tilde{e}_i^* \mathrm{FT}_{ \psi , b_i}  \mathcal F_i$ is a sheaf on a variety of dimension $m-n$, so its shift by $m-n$ is semiperverse, as desired.
\end{proof}

To study $R^j \pi_! \overline{\mathcal F}$ beyond the generic point, we use vanishing cycles, which requires us to compactify the fibers of $\pi$.
Denote by $v \colon \mathbb A^n \to \mathbb P^n$ the usual compactification (of the fiber over $0$).

Let $\mathbb P^n \times \mathbb A^{m-n}$ have projective coordinates $(\lambda: a_0 : \dots : a_{n-1})$ and affine coordinates $a_n, \dots, a_{m-1}$. 
Let $u \colon \mathbb A^m \to \mathbb P^n \times \mathbb A^{m-n}$ send $(a_0,\dots, a_{m-1})$ to 
$
\left( (1: a_0 : \dots : a_{n-1}), (a_n, \dots, a_{m-1}) \right)
$
and denote by $\overline{\pi} \colon \mathbb P^n \times \mathbb A^{m-n} \to \mathbb A^{m-n}$ the projection onto the second factor. 
Then  $\pi= \overline{\pi} \circ u$ and $\overline{\pi}$ is proper while $u$ is an open immersion so for every integer $j$ we have
\begin{equation} \label{AddingAbarToPi}
R^j \pi_! \overline{\mathcal F} = R^j \overline{\pi}_*  u_! \overline{\mathcal F}.
\end{equation}

Let $A$ be a discrete valuation ring over $\overline{\mathbb F_q}$ (with residue field $\overline{\F_q}$), for instance the ring of power series $A = \overline{\F_q}[[\theta]]$.
We have the special point $s \colon \operatorname{Spec} \overline{\F_q} \to \operatorname{Spec} A$ obtained by sending a uniformizer $\theta$ to $0$. 
Let $\tau\colon \operatorname{Spec} A \to \mathbb A^{m-n}$ be any map sending the special point of $\operatorname{Spec} A$ to $0 \in \mathbb{A}^{m-n}$ and the generic point $\eta$ of $\operatorname{Spec} A$ to the generic point $\xi$ of $\mathbb{A}^{m-n}$.
Abusing notation, we will at times view $\tau$ also as a ring homomorphism from $\overline{\F_q}[a_{n}, \dots, a_{m-1}]$ to $A$.
Expanding \cref{SmallFiberCartesianSquare}, we consider the diagram
\begin{equation} \label{CurvedCartesianDiag}
\begin{tikzcd}  
\mathbb A^n \arrow[r,"s_{\mathbb{A}}"]  \arrow[dd, bend right, ""'] \arrow[d, "v"] \arrow[rr, bend left = 22, hook, ""] \arrow[rrr, bend left = 35, "e_i"] & \mathbb A^n \times \operatorname{Spec} A \arrow[d,"u^A"] \arrow[r, "\mathrm{id} \times \tau"] \arrow[rr, bend left = 22, "\overline{e}_i^A"] 
& \mathbb A^{n} \times \mathbb{A}^{m-n} \arrow[d, "u = v \times \textup{id}"]  \arrow[dd, bend left = 72, "\pi"] \arrow[r,"\overline{e}_i"] & \mathbb{A}^{1} \\ 
\mathbb P^n \arrow[r,"s_{\mathbb{P}}"]  
\arrow[d, ""] 
& \mathbb P^n \times \operatorname{Spec} A \arrow[d, "{\overline{\pi}}^A"] \arrow[r, "\mathrm{id} \times \tau"] 
& \mathbb P^n \times \mathbb{A}^{m-n}  \arrow[d,"\overline{\pi}"] \\ 
\operatorname{Spec} \overline{\F_q} \arrow[rr, bend right,"0"] \arrow[r,"s"] & \operatorname{Spec} A \arrow[r,"\tau"] 
& \mathbb{A}^{m-n}
\end{tikzcd}
\end{equation}
where all the new maps are defined in the natural way that makes the diagram commutative and the squares Cartesian.
All the maps and fiber products are over $\operatorname{Spec} \overline{\F_q}$, and we will use abbreviations such as $\mathbb{A}^n_A$ (for $\mathbb A^n \times \operatorname{Spec} A$). We put
\[
\overline{ \mathcal F}^A = \bigotimes_{i=1}^m \overline{e}_i^{A*} \mathcal F_i. 
\]

Viewing $e_i$ as a rational map from $\mathbb P^n$ to $\mathbb P^{1}$, we denote by $U_i$ its (largest) domain of definition, and denote by $\tilde{e}_i \colon U_i \to \mathbb P^{1}$ the resulting morphism.   
Explicitly, the open subset $U_i$ of $\mathbb P^n$ consists of those points $(\lambda: a_0 : \dots : a_{n-1})$ where either $\lambda \neq 0$ or $\sum_{j=0}^{n-1} a_j x_i^j \neq 0$, and 
\[
\tilde{e}_i (\lambda :a_0 : \dots :a_{n-1}) = (\lambda : \sum_{j=0}^{n-1} a_jx_i^j).
\]

\begin{lem}\label{complicated-vanishing} 
The vanishing cycles
$R \Phi u^A_! \overline{\mathcal F}^A$ are supported on the locus
\[ 
X = \left\{  y \in \mathbb P^n : \#\{1 \leq i \leq m : y \notin U_i \textup{ or } \mathcal F_i \textup{ is not translation-invariant at } \tilde{e}_i(y) \} > n \right\}.
\] 

\end{lem}

\begin{proof} 

Fix $y \in \mathbb P^n \setminus X$ and choose a degree $n$ monic polynomial $\overline{g} \in \overline{\F_q}[u]$ divisible by 
\[ \prod_{\substack{ 1 \leq i \leq m \\ y \notin U_i \textrm{ or } \mathcal F_i \textrm{ is not translation-invariant at } \tilde{e}_i(y) }} (u-x_i).\]
Since the leading coefficient of $\overline g$ is a unit of $A$, we can use polynomial long division to find a (necessarily unique) polynomial $t \in A[u]$ of degree less than $m$ that is divisible by $\overline{g}$ and has $\tau(a_n),\dots, \tau(a_{m-1})$ as its coefficients in degrees $n,\dots, m-1$.
As $\tau(a_n), \dots, \tau(a_{m-1})$ lie in the maximal ideal of $A$, the polynomial $t$ vanishes in the residue field of $A$.

We choose a new projective coordinate system for $\mathbb P^n_A$, keeping the coordinate $\lambda$ and defining $b_0,\dots, b_{n-1}$ to be the coefficients of the polynomial
\[a_0 + a_1 u + \dots + a_{n-1} u^{n-1} + \lambda ( \tau(a_n) u^n + \dots + \tau(a_{m-1 }) u^{m-1}) \mod \overline{g}\]
in degrees $0, \dots, n-1$ so that
\[ a_0 + a_1 u + \dots + a_{n-1} u^{n-1} + \lambda ( \tau(a_n) u^n + \dots + \tau(a_{m-1 }) u^{m-1} )= b_0 + b_1 u + \dots + b_{n-1} u^{n-1} + \lambda \overline{g} t. \]

Let $\textup{pr}_b \colon \mathbb P^{n}_A \to \mathbb P^n$ be the projection given by the coordinates $(\lambda : b_0 : \dots : b_{n-1})$.
The remainder of the proof will be devoted to checking that there is an \'{e}tale neighborhood of $s_{\mathbb P}(y)$ on which we have an isomorphism of pullbacks
\begin{equation} \label{RemainderProofIsom}
u^A_! \overline{ \mathcal F}^A \cong \textup{pr}_b^* v_! \mathcal F.
\end{equation}
Then by \cite[Th. Finitude, Corollary 2.16]{sga4h}, the sheaf $v_! \mathcal{F}$ (or any other sheaf) on $\mathbb P^n$ is universally locally acyclic with respect to the constant map $\mathbb{P}^n \to \Spec \overline{\mathbb F_q}$, so $\textup{pr}_b^* v_! \mathcal F$ is locally acyclic with respect to the constant map $\mathbb{P}^n_A \to \Spec A$, and thus $R \Phi \textup{pr}_b^* v_! \mathcal F = 0$ in particular $(R \Phi \textup{pr}_b^* v_! \mathcal F)_y =0$.
Since vanishing cycles are \'{e}tale-local, and $u^A_! \overline{ \mathcal F}^A \cong \textup{pr}_b^* v_! \mathcal F$ in an \'etale neighborhood of $s_{\mathbb P}(y)$, we get that
\[
(R \Phi u^A_! \overline{ \mathcal F}^A)_y = (R \Phi \textup{pr}_b^* v_! \mathcal F)_y =0
\]
as desired.

Because tensor products commute with extensions by zero and pullbacks, in order to prove \cref{RemainderProofIsom} it suffices to find for each $1 \leq i \leq m$ an \'etale neighborhood of $s_{\mathbb P}(y)$ over which there is an isomorphism
\[ u^A_!  \overline{e}_i^{A*} \mathcal F_i \cong \textup{pr}_b^* v_!  e_i^* \mathcal F_i .\]

We split into two cases. In the first case, either $y\notin U_i$ or $\mathcal F_i$ is not translation-invariant at $\tilde{e}_i(y)$, so that $\overline{g}$ vanishes at $x_i$. In this case, the sheaves $u^A_!  \overline{e}_i^{A*} \mathcal F_i $ and $ \textup{pr}_b^* v_!  e_i^* \mathcal F_i $ are simply isomorphic (everywhere), without the need to pass to an \'{e}tale neighborhood. To see this note that 
\begin{equation} \label{CurvedCartesianDiag}
\begin{tikzcd}  
\mathbb A^n_A \arrow[r,"u^A"]  \arrow[dd, "\overline{e}_i^A"', bend right] \arrow[d, ""] & \mathbb P^n_A \arrow[d,"\textup{pr}_b"] \\ 
\mathbb A^n \arrow[r,"v"]  \arrow[d, "e_i"] 
& \mathbb P^n \\ 
\mathbb A^{1}
\end{tikzcd}
\end{equation}
is a commutative diagram with a Cartesian square. 
The proper (or smooth) base change theorem then tells us that $u^A_!  \overline{e}_i^{A*} \mathcal F_i \cong \textup{pr}_b^* v_!  e_i^* \mathcal F_i$ so we are done with the first case.

In the second case, we assume that $y \in U_i$ and that $\mathcal F_i$ is translation-invariant at $\tilde{e}_i(y)$. 
We put $U_{i,A} = \textup{pr}_b^{-1}(U_i)$, and denote by $w \colon U_{i,A} \to \mathbb P^n_A$ the corresponding open immersion.
Define
\[
( \tilde{e}_i, \delta) \colon U_{i,A} \to \mathbb P^{1} \times \mathbb A^{1}, \quad (\tilde{e}_i, \delta)(\lambda : b_0 : \dots: b_{n-1}) = 
((\lambda : b_0 + b_1 x_i + \dots + b_{n-1} x_i^{n-1}),   \overline{g}(x_i) t(x_i) ).\]
Denoting by $\upsilon$ the inclusion of $\mathbb A^{1}$ into $\mathbb P^{1}$, we claim that
\begin{equation} \label{IsomFromCommDiag}
w^* u^A_!  \overline{e}_i^{A*} \mathcal F_i \cong (\tilde{e}_i, \delta)^* \textup{trans}^* \upsilon_{!} \mathcal F_i.
\end{equation}

To prove the claim, we will check that 
\begin{equation} \label{CommDiagCartesSquareTwoTriangles}
\begin{tikzcd} 
& &  \mathbb P^{1} \times \mathbb A^{1}   \arrow[d,"\textup{trans}"]\\ 
 \mathbb P^n_A 
& U_{i,A} \arrow[ur, "{(\tilde{e}_i,\delta)}"]\arrow[l,"w"]  \arrow[r] & \mathbb P^{1}   \\
& \mathbb A^n_A \arrow[u] \arrow[ul, "u^A"] \arrow[r,"\overline{e}_i^A"] &  \mathbb A^{1} \arrow[u,"\upsilon"]
\end{tikzcd}
\end{equation}
is a commutative diagram with Cartesian square.
From this diagram, the isomorphism in \cref{IsomFromCommDiag} follows by applying functoriality of pullback through the top triangle, a trivial form of proper base change through the Cartesian square, the fact that $w^* w_!$ is the identity since $w$ is an open immersion, and functoriality of extension by zero through the left triangle.

The commutativity of the diagram in \cref{CommDiagCartesSquareTwoTriangles} is the statement that, for $(b_0,\dots, b_{n-1}) \in \mathbb A^n_A$ we have
\[
\upsilon (  \overline{e}_i^A  ( b_0,\dots, b_{n-1}) ) = \textup{trans} ( (\tilde{e}_i,\delta) ( 1: b_0 : \dots : b_{n-1} ) ) 
\]
which follows from the calculation
\begin{equation*}
\begin{split}
\textup{trans} ( (\tilde{e}_i,\delta) ( 1: b_0 : \dots : b_{n-1})) &= \textup{trans} ( (1 : b_0 + b_1 x_i + \dots + b_{n-1} x_i^{n-1} ),   \overline{g}(x_i) t(x_i)) \\
&= (1 : b_0 + b_1 x_i+ \dots + b_{n-1} x_i^{n-1}  + 1 \cdot \overline{g}(x_i) t(x_i)).
\end{split}
\end{equation*} 
and the calculation
\begin{equation*}
\begin{split}
\upsilon (  \overline{e}_i^A  ( b_0,\dots, b_{n-1})) &= \upsilon( b_0 + b_1 x_i + \dots + b_{n-1} x_i^{n-1} + 1 \cdot \overline{g}(x_i) t(x_i)) \\
&= (1 : b_0 + b_1 x_i + \dots + b_{n-1} x_i^{n-1} + 1 \cdot \overline{g}(x_i) t(x_i)).
\end{split}
\end{equation*} 

Here and below we use the easily checked Zariski density of the $A$-rational points to deduce that two maps are equal if they agree on all $A$-points, 
although one could also reformulate our argument to be about $R$-points of these rings for an arbitrary $A$-algebra $R$, or to work directly with functions.

To check that the square in \cref{CommDiagCartesSquareTwoTriangles} is Cartesian, 
we note that the vertical arrows are open immersions whose image is the open set where $\lambda \neq 0$, 
so we need to check that the inverse image under $\textup{trans}  \circ (\tilde{e}_i,\delta)$ of the open subset of $\mathbb P^{1}$ where $\lambda \neq 0$ is the open subset of $U_{i,A}$ where $\lambda \neq 0$. This is true because neither $\textup{trans}$ nor $(\tilde{e}_i, \delta)$ changes $\lambda$.
This completes the proof of the isomorphism in \cref{IsomFromCommDiag}.

Next we denote by $\pi_1 \colon \mathbb P^{1} \times \mathbb A^{1} \to \mathbb P^{1}$ the projection on the first coordinate and claim that
\begin{equation} \label{SecondIsomClaimDiagComm}
w^* \textup{pr}_b^* v_! e_i^* \mathcal F_i \cong (\tilde{e}_i, \delta)^* \pi_1^* \upsilon_{!} \mathcal F_i.
\end{equation}
To prove this claim, we look at the commutative diagram 
\begin{equation} \label{CommDiagCartesSquareThreeSquares}
\begin{tikzcd}  \mathbb P^n_A \arrow[d,"\textup{pr}_b"] 
&  U_{i,A} \arrow[d] \arrow[r, "{(\tilde{e}_i,\delta)}"]\arrow[l,"w"] 
&  \mathbb P^{1} \times \mathbb A^{1}   \arrow[d,"\pi_1"]\\ 
\mathbb P^n  
&  U_{i}\arrow[l, "z"]  \arrow[r,"\tilde{e}_i"]
& \mathbb P^{1}   \\
& \mathbb A^n  \arrow[ul,"v"] \arrow[u] \arrow[r,"e_i"] 
 &  \mathbb A^{1}  \arrow[u,"\upsilon"]
 \end{tikzcd} 
 \end{equation}
with Cartesian squares.
From this diagram, the isomorphism in \cref{SecondIsomClaimDiagComm} follows by applying functoriality of pullback through the top-right square, a trivial form of proper base change through the bottom-right Cartesian square, the fact that $z^* z_!$ is the identity since $z$ is an open immersion, functoriality of extension by zero through the bottom-left triangle, and functoriality of pullback through the top-left square.

The commutativity of the top-right square in \cref{CommDiagCartesSquareThreeSquares} is the statement that
\[ \tilde{e}_i ( \lambda: b_0: \dots: b_{n-1}) = \pi_1 ( (\tilde{e}_i,\delta) ( \lambda: b_0 : \dots : b_{n-1})), \quad (\lambda: b_0: \dots:  b_{n-1}) \in U_{i,A} \]
which holds because
\begin{equation*}
\begin{split}
\pi_1 ( (\tilde{e}_i,\delta) ( \lambda: b_0 : \dots : b_{n-1})) &= 
\pi_1 ( (\lambda : b_0 + b_1 x_i + \dots + b_{n-1} x_i^{n-1}), \overline{g}(x_i) t(x_i)) \\
&=  (\lambda : b_0 + b_1 x_i+ \dots + b_{n-1} x_i^{n-1}) = \tilde{e}_i ( \lambda: b_0: \dots : b_{n-1}).
\end{split}
\end{equation*}

The commutativity of the bottom-right square in \cref{CommDiagCartesSquareThreeSquares} is the statement that  
\[ \tilde{e}_i ( b_0 , \dots , b_{n-1}) = \upsilon ( e_i ( b_0 , \dots , b_{n-1})), \quad b_0,\dots, b_{n-1} \in \overline{\mathbb F_q}, \]
which follows from the calculation
\begin{equation*}
\begin{split}
\upsilon (e_i( b_0 , \dots , b_{n-1})) &= \upsilon(b_0 + b_1 x_i+ \dots + b_{n-1} x_i^{n-1}) \\
&= (1: b_0 + b_1 x_i + \dots + b_{n-1} x_i^{n-1})  = \tilde{e}_i ( b_0 , \dots , b_{n-1}).
\end{split}
\end{equation*}

All the squares in \cref{CommDiagCartesSquareThreeSquares} are Cartesian, but we have only used and thus will only check that the bottom-right square is Cartesian.
To check this we note that both vertical arrows are open immersions whose image is the open subset where $\lambda \neq 0$, and that pullback by $\tilde{e}_i$ preserves this open subset as it does not `touch' $\lambda$.

Since $w$ is an open immersion and thus \'{e}tale, it suffices to produce an isomorphism between $w^* u^A_!  \overline{e}_i^{A*} \mathcal F_i $ and $w^* \textup{pr}_b^* v_! e_i^* \mathcal F_i$ over an \'{e}tale neighborhood of $s_{\mathbb P}(y)$. In view of \cref{IsomFromCommDiag} and \cref{SecondIsomClaimDiagComm} we need to produce an isomorphism between $ (\tilde{e}_i, \delta)^* \textup{trans}^* \upsilon_{!} \mathcal F_i$ and  $ (\tilde{e}_i, \delta)^* \pi_1^* \upsilon_{!} \mathcal F_i$ over an \'{e}tale neighborhood of $s_{\mathbb P}(y)$. 
As we are in the second case, the sheaf $\mathcal F_i$ is translation-invariant at $\tilde{e}_i(y)$. 
Thus, by the definition of translation-invariance, there is an isomorphism between $\pi_1^* \upsilon_{!} \mathcal F_i$ and $\textup{trans}^* \upsilon_{!} \mathcal F_i$ in some \'etale neighborhood of $(\tilde{e}_i(y),0)$. 
As
\[
(\tilde{e}_i, \delta)(s_{\mathbb P}(y)) = (\tilde{e}_i(\textup{pr}_b(s_{\mathbb P}(y))),0) = (\tilde{e}_i(y),0),
\]
pulling back that \'{e}tale neighborhood and isomorphism under $(\tilde{e}_i, \delta)$ gives the claim.
\end{proof}

\begin{lem} \label{weird-set-dimension} 

For every $1 \leq i \leq m$ suppose that all the slopes of the local monodromy representation of $\mathcal F_i$ at $\infty$ are bounded from above by $1$. 
Then the dimension of $X$ is $0$. 
\end{lem}

\begin{proof} 

The locus $X$ is the union over all subsets $T$ of $\{1,\dots, m\}$ of cardinality exceeding $n$ of the locus $X_T$ of all those $y \in \mathbb P^n$ where $y\notin U_i$ or $\mathcal F_i$ is not translation-invariant at $\tilde{e}_i(y)$ for all $i \in T$. It suffices to prove that $X_T$ is $0$-dimensional for all such $T$.

By \cref{AffineTranslationInvariance}, \cref{TranslationInvarianceOnTheLine}, the fact that $\mathcal F_i$ is lisse away from finitely many points (by constructibility),
and our assumption that the slopes of $\mathcal F_i$ at $\infty$ are at most $1$, the set of $y \in U_i$ for which $\mathcal F_i$ is not translation-invariant at $\tilde{e}_i(y)$ is the union over the finitely many singular points $z \in \mathbb A^1$ of $\mathcal F$ of $\tilde{e}_i^{-1}(z)$, which, using $(\lambda : f)$ as coordinates of $\mathbb P^n$, is the hyperplane $f(x_i) = \lambda z$. The set of $y \in \mathbb P^n \setminus U_i$ is contained in the projective hyperplane $f(x_i) =  0$.
Thus, among all $(\lambda:f)$ in the affine cone over $X_T$, there are only finitely many possibilities for $f(x_i)$ for each $\lambda$, so by the Chinese Remainder Theorem there are only finitely many possibilities for $f \mod \prod_{i \in T} (u-x_i)$ for each $\lambda$. 

Since $f$ is a polynomial of degree less than $n$, and the degree of $\prod_{i \in T} (u-x_i)$ is $\abs{T} \geq n$, there is at most one polynomial $f$ in each congruence class modulo $\prod_{i \in T} (u-x_i)$, so the dimension of $X_T$ is indeed $0$.
\end{proof}

\begin{cor} \label{concentration-lemma} 

For every $1 \leq i \leq m$ suppose that all the slopes of the local monodromy representation of $\mathcal F_i$ at $\infty$ are bounded from above by $1$. 
Then for every integer $j > n$ we have $H^j(\mathbb P^n, R \Phi u^A_! \overline{ \mathcal F}^A) = 0$. 

\end{cor} 

\begin{proof}  

As $ u^A_! \overline{\mathcal F}^A$ is a sheaf on a scheme of dimension $n+1$, its shift $u^A_! \overline{\mathcal F}^A[n+1]$ is semiperverse. 
By the semiperversity of vanishing cycles (which follows from \cite[Proposition 4.4.2]{BBD} using the vanishing cycles long exact sequence), we know that $R\Phi u^A_! \overline{\mathcal F}^A [n]$ is semiperverse. 

By \cref{complicated-vanishing} and \cref{weird-set-dimension}, the support of $ R \Phi   u^A_!  \overline{\mathcal F}^A$ has dimension at most $n$.
This in conjunction with semiperversity and \cite[4.2.3]{BBD} tells us that the cohomology of $ R \Phi   u^A_! \overline{ \mathcal F}^A$ is concentrated in degrees at most $n$.
\end{proof} 

\begin{prop} \label{tau-pullback-equation} 
For every integer $j$ we have
$
\tau^* R^j \overline{\pi}_*  u_! \overline{ \mathcal F} =  R^j \overline{\pi}^A_*  u^A_! \overline{ \mathcal F}^A.
$
\end{prop}

\begin{proof}

From proper base change over the lower-right (Cartesian) square in \cref{CurvedCartesianDiag} we get
\begin{equation*} 
\tau^* R^j \overline{\pi}_*  u_! \overline{ \mathcal F} =  
R^j \overline{\pi}^A_* (\mathrm{id} \times \tau)^* u_! \overline{ \mathcal F}
\end{equation*}
and proper base change over the upper-right (Cartesian) square in \cref{CurvedCartesianDiag} gives
\begin{equation*} 
R^j \overline{\pi}^A_* (\mathrm{id} \times \tau)^* u_! \overline{ \mathcal F} = 
R^j \overline{\pi}^A_* u^A_! (\mathrm{id} \times \tau)^* \overline{ \mathcal F}.
\end{equation*}
The result now follows from the definition $\overline{e}_i^A = \overline{e}_i \circ (\mathrm{id} \times \tau)$.
\end{proof}

\begin{cor} \label{TechHeart}

For every $1 \leq i \leq m$ suppose that $\mathcal F_i$ has no finitely supported sections, that its (geometric) global monodromy has neither Artin--Schreier nor trivial factors, and that all the slopes of the local monodromy representation of $\mathcal F_i$ at $\infty$ are bounded from above by $1$. 
Then for every integer $j > n+1$ we have
$
H^j_c ( \mathbb A^n,  \mathcal F ) = 0. 
$

\end{cor}

\begin{proof}

We have already identified
$H^j_c ( \mathbb A^n, \mathcal F) $ with the stalk of $R^j \pi_! \overline{ \mathcal F}$ at $0$,
which we further identify with the stalk of $R^j \overline{\pi}_*  u_! \overline{ \mathcal F}$ at $0$ using \cref{AddingAbarToPi}.
We can also view the latter as the stalk of $\tau^* R^j \overline{\pi}_*  u_! \overline{ \mathcal F}$ 
at the special point $s$ of $\operatorname{Spec} A$.
By \cref{tau-pullback-equation}, we have 
\[( \tau^* R^j \overline{\pi}_*  u_! \overline{\mathcal F} )_{s} = (R^j \overline{\pi}^A_*  u^A_! \overline{\mathcal F}^A )_{s}.\] 
This reasoning also shows that
\begin{equation} \label{TwoGenericStalks}
(R^j \pi_! \overline{ \mathcal F})_{\overline{\xi}} = (R^j \overline{\pi}^A_*  u^A_! \overline{\mathcal F}^A )_{\overline \eta}
\end{equation}
where $\overline{\xi}$ (respectively, $\overline{\eta}$) is a geometric generic point of $\mathbb{A}^{m-n}$ (respectively, $\Spec A$).

We have a vanishing cycles long exact sequence
\[ \dots \to (R^j \overline{\pi}^A_*  u^A_! \overline{\mathcal F}^A )_{s} \to (  R^j \overline{\pi}^A_*  u^A_! \overline{\mathcal F}^A)_{\overline \eta}  \to  H^j ( \mathbb P^n,  R \Phi   u^A_! \overline{\mathcal F}^A) \to \dots . \] 
By \cref{generic-step} and \cref{TwoGenericStalks} we have $(  R^j \overline{\pi}^A_*  u^A_! \overline{\mathcal F}^A)_{\overline \eta} = 0$ for $j > n$.
By \cref{concentration-lemma} the third term is also concentrated in degrees at most $n$, so the first term is concentrated in degrees at most $n +1$, as desired. 
\end{proof}

\section{Proof of \cref{MainRes}}

Denote by $m$ the degree of $g$, let $x_1, \dots, x_m \in \overline{\F_q}$ be the distinct roots of $g$, and let $n \leq m$ be a positive integer. 
We view $\mathbb A^n$ as parametrizing polynomials in $\F_q[u]$ of degree less than $n$ via their coefficients.
This space will be used to parametrize the short intervals over which the summation in our exponential sums is performed.
For every $1 \leq i \leq m$ denote by $e_i \colon \mathbb A^n \to \mathbb A^{1}$ the (linear) map that evaluates a polynomial in $\overline{\F_q}[u]$ at $x_i$, and note that there exists a unique monic irreducible factor $\pi$ of $g$ with $\pi(x_i) = 0$.
We define the sheaf $\mathcal{F}_i$ on $\mathbb{A}^1$ over $\overline{\F_q}$ to be the base change of $\mathcal{F}_\pi$ along the embedding
$\mathbb F_q[u]/(\pi) \hookrightarrow \overline{ \mathbb F_{q}}$ mapping $u$ to $x_i$.
That is, we have the Cartesian square
\begin{equation*} 
\begin{tikzcd}  
\mathbb A^{1}_{\overline {\F_q}} \arrow[d,""] \arrow[r, "\xi"]  & \mathbb A^{1}_{\F_q[u]/(\pi)} \arrow[d, ""] \\
\operatorname{Spec} \overline{\F_q} \arrow[r, "u \mapsto x_i"] &\operatorname{Spec} \F_q[u]/(\pi) 
\end{tikzcd}
\end{equation*}
and we set $\mathcal F_i = \xi^* \mathcal F_\pi$.
Our assumption that $\mathcal F_\pi$  has no finitely supported sections means that neither does $\mathcal F_i$.
Our assumption that the global geometric monodromy of $\mathcal F_\pi$  has neither Artin--Schreier nor trivial factors means that the same holds also for $\mathcal F_i$.
Our assumption that all the slopes of the geometric local monodromy of $\mathcal F_\pi$ at $\infty$ are bounded from above by $1$ means that the same holds also for $\mathcal F_i$.

It is shown in the proof of \cite[Corollary 3.14]{SS20} that the sheaf
\[
\mathcal F = \bigotimes_{i=1}^m e_i^* \mathcal F_i
\]
on $\mathbb A^n_{\overline{\F_q}}$ descends to $\F_q$, is mixed of nonpositive weights, and that its trace function satisfies
\[
\operatorname{tr}(\operatorname{Frob}_q, \mathcal F_{\bar f}) = t(f)
\]
so for the sum we need to obtain cancellation in we have
\[
\sum_{\substack{f \in \F_q[u] \\ \deg(f) < n}} t(f) = \sum_{f \in \mathbb A^n(\F_q)} \operatorname{tr}(\operatorname{Frob}_q, \mathcal F_{\bar f}).
\]
Specifically, for $q^{n-1} < X  \leq q^{n}$, the sum on the left hand side above equals the sum in \cref{MainRes}.

The Grothendieck--Lefschetz trace formula tells us that
\begin{equation*}
\sum_{f \in \mathbb A^n(\F_q)} \operatorname{tr}(\operatorname{Frob}_q, \mathcal F_{\bar f}) = 
\sum_{j=0}^{2n} (-1)^j \operatorname{tr}(\operatorname{Frob}_q,H_c^j(\mathbb A^n, \mathcal F))
\end{equation*}
where (the compactly supported) \'etale cohomology is taken over $\overline{\F_q}$.
Since $\mathcal F$ is mixed of nonpositive weights it follows from \cite{Weil2} that for each $0 \leq j \leq 2n$, every eigenvalue of $\operatorname{Frob}_q$ on $H_c^j(\mathbb A^n, \mathcal F)$ is of absolute value at most $q^{j/2}$. 
Applying the triangle inequality to the sum over $j$ we thus get that
\begin{equation*} \label{AfterDeligneEq}
\left| \sum_{j=0}^{2n} (-1)^j \operatorname{tr}(\operatorname{Frob}_q,H_c^j(\mathbb A^n, \mathcal F)) \right| \leq \sum_{j=0}^{2n} q^{j/2} \cdot \dim_{\overline{\mathbb Q_\ell}} H_c^j(\mathbb A^n, \mathcal F).
\end{equation*}
\cref{TechHeart} tells us that there can only be a contribution to the right hand side from $j \leq n+1$, so we arrive at an upper bound of
\[
q^{\frac{n+1}{2}} \sum_{j=0}^{2n}  \dim_{\overline{\mathbb Q_\ell}} H_c^j(\mathbb A^n, \mathcal F).
\]

\cite[Lemma 3.13]{SS20} tells us that the above sum of Betti numbers does not exceed the coefficient of $Z^n$ in the real polynomial
\[
  \prod_{i=1}^m( \rank(\mathcal F_i) (1+Z) + c_F(\mathcal F_i) Z)
\]
where $\mathbf{r}$ is the generic rank of a sheaf, and $c_F$ is the Fourier conductor of a sheaf introduced in \cite[Definition 3.8]{SS20}.
By the argument on \cite[page 67]{SS20} the above equals the coefficient of $Z^n$ in the polynomial
\[
\prod_{\pi \mid g} \left( r(t_{\pi} )(1+Z)  + c(t_{\pi}) Z \right)^{\deg (\pi)}
\]
where $\pi$ ranges over the monic irreducible factors of $g$.
That argument relies on \cite[Lemma 3.10 (6)]{SS20} which assumes that the sheaves $\{\mathcal{F}_i\}_{i=1}^m$ are tamely ramified at $\infty$, an assumption not necessarily satisfied in our case, but the proof of that lemma remains valid under our weaker assumption that the slopes at $\infty$ are at most $1$.
The bound in \cref{MainRes} can be deduced from the above using the argument on \cite[page 68]{SS20}, so the proof is complete.

\section{Proof of \cref{GeneralizedHooleyRationalFunction}}

In view of the Chinese Remainder Theorem, for every monic irreducible factor $\pi$ of $g$ there exists a (unique) nontrivial character $\chi_{\pi} \colon (\F_q[u]/(\pi))^\times \to \mathbb C^\times$ such that
\[
\chi = \prod_{\pi \mid g} \chi_\pi. 
\]
In particular for every $h \in \F_q[u]$ we have
\[
\chi(F(u,h)) = \prod_{\pi \mid g} \chi_\pi(F(u,h)).
\]
\cite[Lemma 2.2]{SS20} provides us with a tame, mixed of nonpositive weights, sheaf $\mathcal L_{\chi_\pi}(F)$ on $\mathbb A^1_{\F_q[u]/(\pi)}$ whose trace function at $h \in \F_q[u]$ is $\chi_{\pi}(F(u,h))$, and it has no finitely supported sections.
Moreover, its rank is $1$ and its conductor is bounded from above by the degree (in $T$) of the reduction of $F$ moduli $\pi$ in case this reduction is nonzero (and otherwise the conductor is $0$).
In particular $\chi(F(u,-))$ is a trace function.

By B\'ezout's Lemma, for every monic irreducible factor $\pi$ of $g$ there exists a polynomial $f_\pi \in \F_q[u]$ not divisible by $\pi$ such that
\[
\sum_{\pi \mid g} f_\pi \cdot \frac{g}{\pi} = 1.
\]
We therefore have
\[
\begin{split}
e\left( \frac{a(u,h) \overline{b(u,h)}}{g} \right) &= e\left( \frac{a(u,h) \overline{b(u,h)}}{g} \cdot \sum_{\pi \mid g} f_\pi \cdot \frac{g}{\pi}  \right) =  
e\left( \sum_{\pi \mid g} \frac{ f_\pi \cdot a(u,h) \overline{b(u,h)}}{\pi}   \right) \\ &= 
\prod_{\pi \mid g} e\left( \frac{ f_\pi \cdot a(u,h) \overline{b(u,h)}}{\pi}   \right). 
\end{split}
\]
\cite[Lemma 2.15]{SS20} provides us with a sheaf $\mathcal L_{\psi}(f_\pi a/b)$ over $\mathbb A^1_{\F_q[u]/(\pi)}$ mixed of nonpositive weights, having no finitely supported sections, and whose trace function is the factor corresponding to $\pi$ in the product above. 
Moreover, the rank of $\mathcal L_{\psi}(f_\pi a/b)$ is $1$, and our assumption on the degrees (in $T$) of the reductions of $a$ and $b$ modulo $\pi$ guarantess that the slope of $\mathcal L_{\psi}(f_\pi a/b)$ at $\infty$ is at most $1$, and that the conductor of $\mathcal L_{\psi}(f_\pi a/b)$ is at most $2 \deg_T b$.
In particular the function
\[
h \mapsto e\left( \frac{a(u,h) \overline{b(u,h)}}{g} \right)
\]
is a trace function.

\cite[Proposition 2.11, Lemma 2.13 and its proof]{SS20} tell us that $\mathcal L_{\chi_\pi}(F) \otimes \mathcal L_{\psi}(f_\pi a/b)$ is a sheaf mixed of nonpositive weights, has no finitely supported sections, has rank $1$, its slope at $\infty$ is at most $1$, its trace function is
\[
\chi_\pi(F(u,h)) \cdot e\left( \frac{ f_\pi \cdot a(u,h) \overline{b(u,h)}}{\pi}   \right)
\]
and its conductor is at most $\max\{\deg_T F + 2 \deg_T b,  2 \deg_T b\}$.
In particular the function of $h$ summed in \cref{GeneralizedHooleyRationalFunction} is a trace function.

To complete the proof of \cref{GeneralizedHooleyRationalFunction} we invoke \cref{MainRes}.
This requires checking that the sheaves $\mathcal L_{\chi_\pi}(F) \otimes \mathcal L_{\psi}(f_\pi a/b)$ have no geometric Artin--Schreier or trivial factors. We check this first in case the degree (in $T$) of the reduction of $F$ modulo $\pi$ is positive.
Indeed in this case the sheaf $\mathcal L_{\chi_\pi}(F)$ has a singularity at any zero of $F$, so the same is true for the sheaf $\mathcal L_{\chi_\pi}(F) \otimes \mathcal L_{\psi}(f_\pi a/b)$, but Artin--Schreier sheaves and constant sheaves are lisse on the affine line, so since all the sheaves in question have rank $1$, the treatment of this case is complete.

Suppose now that the reduction modulo $\pi$ of $F$ is constant, so that our sheaf is $\mathcal L_{\psi}(f_\pi a/b)$.
Our assumptions then imply that the rational function $a/b$ is not a polynomial, hence our sheaf has a singularity at some zero of $b$. As in the previous case the required conclusion follows from the fact that Artin--Schreier sheaves and constant sheaves are lisse on $\mathbb A^1$.
The invocation of \cref{MainRes} is thus justified so the proof is complete.

\section{Proof of \cref{MordellCor}}

It follows from the proof of \cite[Proposition 6.7]{Fouvry--Kowalski--Michel14}, 
that there exists a subset $S \subseteq \F_q[u]/(\pi)$ of at most $d-1$ elements, a positive integer $m \leq d!$, 
irreducible tame geometrically nontrivial middle-extension sheaves $\{ \mathcal{F}_i \}_{i=1}^{m}$ on $\mathbb{A}^1_{\F_q[u]/(\pi)}$ with finite monodromy which are punctually pure of weight $0$,
and algebraic numbers $\{c_i\}_{i=1}^{m}$, such that
\begin{equation*} \label{IndicatorDecompositionModPrime}
\mathbf{1}_{\mathcal{P}}(x) = \frac{|\mathcal{P}|}{|\pi|} + \sum_{i=1}^{m} c_i t_{i}(x) + O \left( d!^3 |\pi|^{-1/2} \right), \quad x \in \F_q[u]/(\pi) \setminus S,
\end{equation*}
namely the characteristic function of the set $\mathcal{P}$ of values of $P$ modulo $\pi$ is, away from $S$, a linear combination of the trace functions of the sheaves $\{\mathcal{F}_i\}_{i=1}^{m}$, up to $d!^3|\pi|^{-1/2}$ times an absolute constant. 
Moreover
\[
|c_1|, \dots, |c_{m}| \leq d!, \quad r(t_1), \dots, r(t_{m}) \leq d!, \quad c(t_1), \dots, c(t_{m}) \leq d!.
\]

We check that \cref{MainRes} applies to our sheaves $\mathcal F_1, \dots, \mathcal F_m$.
For this, we first recall that a sheaf punctually pure of weight $0$ is mixed of nonpositive weights.
Second, we recall that middle-extension sheaves have no finitely supported sections.
Third, we recall that tameness of sheaves, means that the slopes at $\infty$ (or any other point in $\mathbb{P}^1$) are $0$.
In particular, our sheaves have no Artin--Schreier quotients. 
It remains to explain why our sheaves do not have the constant rank $1$ sheaf as a geometric quotient.

Suppose toward a contradiction that $\overline{\mathbb{Q}_\ell}$ is geometrically a quotient of $\mathcal{F}_i$ for some $1 \leq i \leq m$. 
Since $\mathcal{F}_i$ is irreducible with finite monodromy, Clifford theory tells us that $\mathcal{F}_i$ is geometrically a direct sum of (say) $n$ constant sheaves of rank $1$.
We conclude that the arithmetic monodromy of $\mathcal{F}_i$ is abelian, so $n=1$ because of its irreducibility. We arrive at a contradiction to the geometric nontriviality of $\mathcal{F}_i$. 

We can therefore invoke \cref{MainRes} and get that
\begin{equation*}
\begin{split}
\# \{f \in \F_q[u] : |f| < X, \ f \in \mathcal{P} \} &= \sum_{\substack{f \in \F_q[u] \\ |f| < X}} \mathbf{1}_{\mathcal{P}}(f) =
\frac{|\mathcal{P}|}{|\pi|}X + \sum_{i=1}^m c_i \sum_{\substack{f \in \F_q[u] \\ |f| < X}} t_i(f) + O\left( d!^4 + \frac{d!^4 X}{|\pi|^{1/2}} \right) \\
&= \frac{|\mathcal{P}|}{|\pi|}X + O \left( X^{1/2} |\pi|^{\log_q(3 \cdot d!)} + d!^4 + \frac{d!^4 X}{|\pi|^{1/2}} \right)
\end{split}
\end{equation*}
where the first error term dominates the second error term, so the result follows.

\section{Covariance of Short Trace Sums}

Our goal here will be to establish a generalized form of \cref{VarianceShortTraceSums}. 
To set it up, we denote the degree of $P$ by $m$, and let $\mathcal{F}, \mathcal{G}$ be sheaves on $\mathbb{A}^1_{\F_q[u]/(P)}$ satisfying the assumptions of \cref{VarianceShortTraceSums}. For $h \in \F_q[u]/(P)$ set
\[
A_h \colon \mathbb{A}^1_{\F_q[u]/(P)} \to \mathbb{A}^1_{\F_q[u]/(P)}, \quad A_h(x) = x-h.
\]

For a geometric generic point $\eta$ of $\mathbb{A}^1_{\F_q[u]/(P)}$ we assume that $\mathcal{F}_\eta$ and $\mathcal{G}_\eta$ are irreducible representations of $\Gal(\overline{\F_q[u]/(P)}(X)^{\text{sep}}/\overline{\F_q[u]/(P)}(X))$, and that for every $h \neq 0$ the representations $\mathcal{F}_\eta$ and $(A_h^* \mathcal{G})_\eta$ are not isomorphic.
We denote the trace functions of $\mathcal{F}$ and $\mathcal{G}$ by $t_1$ and $t_2$ respectively.

For $k \leq m$, we will estimate the covariance of $t_1$ and $t_2$ by
\[
q^{-m} \sum_{f \in \F_q[u]/(P)} \left( q^{-k/2} \sum_{\substack{g \in \F_q[u] \\ \deg(g) < k}} t_1(f+g) \right) 
\overline{ \left( q^{-k/2} \sum_{\substack{h \in \F_q[u] \\ \deg(h) < k}} t_2(f+h) \right) }
= \mathbf{1}_{\mathcal{F}_\eta \cong \mathcal{G}_\eta} + O\left(q^\frac{k-m}{2}|P|^{\log_q(\gamma)} \right) 
\]
where the implied constant depends only on $q$, and
\[
\gamma = 3(r(t_1) + c(t_1))(r(t_2) + c(t_2)).
\]
In case $\mathcal{F} = \mathcal{G}$, we have $t_1 = t_2$ so we obtain \cref{VarianceShortTraceSums} with $X = q^k$.

Expanding the brackets above we get
\[
q^{-m} \sum_{f \in \F_q[u]/(P)} q^{-k} \sum_{\substack{g,h \in \F_q[u] \\ \deg(g), \deg(h) < k}} t_1(f+g) \overline{t_2(f+h)}
\]
and a change of variables (and of the order of summation) brings us to
\begin{equation} \label{InnerSumIsTraceFunctionOfConvolution}
\sum_{\substack{h \in \F_q[u] \\ \deg(h) < k}} q^{-m} \sum_{f \in \F_q[u]/(P)} t_1(f) \overline{t_2(f-h)}.
\end{equation}

Put
\[
s \colon \mathbb{A}^2 \to \mathbb{A}^1, \quad s(f,h) = f+h, \qquad N \colon \mathbb{A}^1 \to \mathbb{A}^1, \quad N(x) = -x
\]
and consider the additive convolution
\[
\mathcal{F} * N^*\mathcal{G}^\vee = Rs_!(\mathcal{F} \boxtimes N^* \mathcal{G}^\vee)
\]
whose trace function at $h \in \F_q[u]/(P)$ is, by the Grothendieck--Lefschetz trace formula, and our assumption that $\mathcal F$ and $\mathcal G$ are punctually pure of weight $0$, the inner sum in \cref{InnerSumIsTraceFunctionOfConvolution}.

By the proper base change theorem, and \cite[2.0.6]{Katz88} for $h \in \F_q[u]/(P)$ we have
\[
(R^2s_!(\mathcal{F} \boxtimes N^* \mathcal{G}^\vee))_h = H^2_c(\mathbb{A}^1, \mathcal{F} \otimes A^*_{h} \mathcal{G}^\vee) = 
\begin{cases}
0 &h \neq 0 \\ 
0 &h = 0, \ \mathcal{F}_\eta \ncong \mathcal{G}_\eta \\
\overline{\mathbb{Q}_\ell} &h=0, \ \mathcal{F}_\eta \cong \mathcal{G}_\eta
\end{cases}
\]
since the representations are irreducible by assumption.
Frobenius acts by multiplication by $q^m$ so the degree $2$ contribution to the trace function of $\mathcal{F} * N^*\mathcal{G}^\vee$ gives the required main term. 
Similarly
\begin{equation} \label{NoContributionTraceConvolutionDegreeZero}
(R^0s_!(\mathcal{F} \boxtimes N^* \mathcal{G}^\vee))_h  = H^0_c(\mathbb{A}^1, \mathcal{F} \otimes A_h^*\mathcal{G}^\vee) = 0
\end{equation}
since the sheaves $\mathcal{F}$ and $A_h^* \mathcal{G}^\vee$ have no finitely supported sections, 
so there is no degree $0$ contribution to the trace function of $\mathcal{F} * N^*\mathcal{G}^\vee$.

In order to study the degree $1$ contribution we define the (twisted) sheaf
\[
\mathcal{H} = R^1s_!(\mathcal{F} \boxtimes N^* \mathcal{G}^\vee)(1/2)
\]
and denote its trace function by $t$, so that
\[
\sum_{\substack{h \in \F_q[u] \\ \deg(h) < k}} q^{-m} \sum_{f \in \F_q[u]/(P)} t_1(f) \overline{t_2(f+h)} = 
\mathbf{1}_{\mathcal{F}_\eta \cong \mathcal{G}_\eta} - q^{-m/2} \sum_{\substack{h \in \F_q[u] \\ \deg(h) < k}} t(h).
\]
Our task is therefore to bound the sum on the right hand side above.
We will do so by invoking \cref{MainRes}. 
For that matter, we need to check that $\mathcal{H}$ satisfies the required properties, and calculate its rank and conductor.
First we note that, in view of Deligne's theorem on weights, 
our choice of twist is such that $\mathcal{H}$ is mixed of nonpositive weights since $\mathcal{F}$ and $\mathcal{G}$ are punctually pure of weight $0$.

We show that $\mathcal{H}$ has no finitely supported sections. Our task is to prove
\[
H^0_c(\mathbb{A}^1, R^1s_!(\mathcal{F} \boxtimes N^* \mathcal{G}^\vee)) = 0.
\]
The Leray spectral sequence for $s$ is 
\[
H^i_c(\mathbb{A}^1, R^js_!(\mathcal{F} \boxtimes N^* \mathcal{G}^\vee)) \Rightarrow 
H^{i+j}_c(\mathbb{A}^2, \mathcal{F} \boxtimes N^* \mathcal{G}^\vee)
\]
and it gives rise to the exact sequence
\[
H^{1}_c(\mathbb{A}^2, \mathcal{F} \boxtimes N^* \mathcal{G}^\vee) \to 
H^0_c(\mathbb{A}^1, R^1s_!(\mathcal{F} \boxtimes N^* \mathcal{G}^\vee)) \to 
H^2_c(\mathbb{A}^1, R^0s_!(\mathcal{F} \boxtimes N^* \mathcal{G}^\vee))
\]
so in order to show that the middle term vanishes, it suffices to show that the others do.
For the last term, it is enough to check that $R^0s_!(\mathcal{F} \boxtimes N^* \mathcal{G}^\vee)$ is the zero sheaf, 
namely that all its stalks vanish, which we have seen in \cref{NoContributionTraceConvolutionDegreeZero}.  

For the first term in the exact sequence above, the K\"unneth formula gives
\begin{equation*}
\begin{split}
H^{1}_c(\mathbb{A}^2, \mathcal{F} \boxtimes N^*\mathcal{G}^\vee) = ( H^1_c(\mathbb{A}^1, \mathcal{F}) \otimes H^0_c(\mathbb{A}^1, N^* \mathcal{G}^\vee) ) \oplus  ( H^0_c(\mathbb{A}^1, \mathcal{F}) \otimes H^1_c(\mathbb{A}^1, N^* \mathcal{G}^\vee) ) = 0
\end{split}
\end{equation*}
since $\mathcal{F}$ and $\mathcal{G}^\vee$ have no finitely supported sections by assumption. 
This concludes the proof that $\mathcal{H}$ has no finitely supported sections.

Using a similar argument, we show that $\mathcal{H}_\eta$ has neither trivial nor Artin--Schreier quotients.
Our task here is to show that
\[
H^2_c(\mathbb{A}^1,R^1s_!(\mathcal{F} \boxtimes N^* \mathcal{G}^\vee) \otimes \mathcal{L}_\psi(ax)) = 0, \quad a \in \overline{\F_q[u]/(P)}.
\]
By the projection formula, this amounts to showing that
\[
H^2_c(\mathbb{A}^1,R^1s_!(\mathcal{F} \boxtimes N^* \mathcal{G}^\vee \otimes \mathcal{L}_\psi(a(f+h)))) = 0.
\]
In view of the Leray spectral sequence for $s$, it is enough to show that
\[
H^3_c(\mathbb{A}^2, \mathcal{F} \boxtimes N^* \mathcal{G}^\vee \otimes \mathcal{L}_\psi(a(f+h))) = 0.
\]
Since $\mathcal{L}_\psi(a(f+h)) = \mathcal{L}_\psi(af) \boxtimes \mathcal{L}_\psi(ah)$, the K\"unneth formula gives
\begin{equation*}
\begin{split}
&H^3_c(\mathbb{A}^2, \mathcal{F} \boxtimes N^* \mathcal{G}^\vee \otimes \mathcal{L}_\psi(a(f+h))) = 
H^3_c(\mathbb{A}^2, (\mathcal{F} \otimes \mathcal{L}_\psi(af)) \boxtimes (N^* \mathcal{G}^\vee \otimes \mathcal{L}_\psi(ah))) = \\
&(H_c^1( \mathbb{A}^1, \mathcal{F} \otimes \mathcal{L}_\psi(af) ) \otimes H^2_c(\mathbb{A}^1, N^* \mathcal{G}^\vee \otimes \mathcal{L}_\psi(ah)) ) \oplus 
(H_c^2(\mathbb{A}^1, \mathcal{F} \otimes \mathcal{L}_\psi(af)) \otimes H^1_c(\mathbb{A}^1, N^* \mathcal{G}^\vee \otimes \mathcal{L}_\psi(ah) ) )
\end{split}
\end{equation*}
and this direct sum vanishes since the cohomology groups in degree $2$ vanish in view of our assumptions on $\mathcal{F}$ and on $\mathcal{G}$.

We explain why $\mathcal{H}$ has all slopes at $\infty$ bounded by $1$. 
By \cite[Proposition 1.2.2.7]{Lau}, the Fourier transform of $\mathcal{F} * N^*\mathcal{G}^\vee$ is the tensor product of the Fourier transform of $\mathcal{F}$ and the Fourier transform of $\mathcal{G}$.
Since the Fourier transform is involutive, and the property of having all slopes at $\infty$ bounded by $1$ is preserved by the Fourier transform in view of \cite[Proposition 2.4.3 (i)(b), (iii)(b)]{Lau}, and by tensor products in view of \cite[Proposition 2.11, Proof of Corollary 2.12]{SS20}, our claim follows from the corresponding assumption on $\mathcal{F}$ and $\mathcal{G}$.

For the rank of $\mathcal H$, by \cite[Lemma 3.11]{SS20} we have
\[
r(t) \leq \sup_{h} \dim H^1_c(\mathbb{A}^1, \mathcal{F} \otimes A_h^*\mathcal{G}^\vee) \leq c(t_1)r(t_2) + r(t_1)c(t_2) + r(t_1)r(t_2).
\]
For the conductor, since all slopes at $\infty$ are at most $1$, by \cite[Lemma 3.10]{SS20} and \cite[Proposition 2.7.2.1]{Lau} we have
\[
c(t) = \sum_{x \in |\mathbb{A}^1|} \cond_x(\mathcal{H}) \leq \sum_{x \in |\mathbb{A}^1|} \cond_x(\mathcal{F} * N^*\mathcal{G}^\vee) =
\sum_{x \in |\mathbb{A}^1|} \sum_{\substack{y,z \in |\mathbb{A}^1| \\ y+z = x}} \cond_y(\mathcal{F}) \cond_z(N^*\mathcal{G}^\vee)
= c(t_1)c(t_2).
\]
We can therefore apply \cref{MainRes} and get the required bound.

\section{Acknowledgments}

Will Sawin's research was supported by the Clay Mathematics Institute, the National Science Foundation grant DMS-2502029, and a Sloan Research Fellowship. 

Mark Shusterman's research is co-funded by the European Union (ERC, Function Fields, 101161909). Views and opinions expressed are however those of the authors only and do not necessarily reflect those of the European Union or the European Research Council. Neither the European Union nor the granting authority can be held responsible for them.

\end{document}